\documentclass[final]{siamltex}
\usepackage{animate}
\usepackage{amsmath}
\usepackage{amssymb}
\usepackage{graphics}
\usepackage{graphicx}
\usepackage{footnote}
\usepackage{textcomp}
\usepackage{mathrsfs}
\usepackage{epstopdf}
\usepackage{array}
\usepackage{showkeys}
\usepackage[maxfloats=99]{morefloats}
\usepackage{url}
\usepackage{cases}
\usepackage{mathscinet}
\usepackage[normalem]{ulem}
\usepackage[ruled,linesnumbered]{algorithm2e}
\usepackage[noend]{algpseudocode}
\usepackage{color}
\usepackage{cite}
\usepackage{bm}
\usepackage{subcaption}

\usepackage{stmaryrd}
\usepackage{color}
\usepackage{mathtools}
\graphicspath{{./fig/em/}{./fig/dr/}{./fig/other/}}

\makeatletter
\def\BState{\State\hskip-\ALG@thistlm}
\makeatother

\numberwithin{equation}{section}
\newtheorem{assumption}[theorem]{Assumption}

\newtheorem{remark}{Remark}[section]

\usepackage{xparse}

\NewDocumentCommand{\dgal}{sO{}m}{%
  \IfBooleanTF{#1}
    {\dgalext{#3}}
    {\dgalx[#2]{#3}}%
}

\NewDocumentCommand{\dgalext}{m}{%
  \sbox0{%
    \mathsurround=0pt % just for safety
    $\left\{\vphantom{#1}\right.\kern-\nulldelimiterspace$%
  }%
  \sbox2{\{}%
  \ifdim\ht0=\ht2
    \{\kern-.45\wd2 \{#1\}\kern-.45\wd2 \}%
  \else
    \left\{\kern-.5\wd0\left\{#1\right\}\kern-.5\wd0\right\}%
  \fi
}

\NewDocumentCommand{\dgalx}{om}{%
  \sbox0{\mathsurround=0pt$#1\{$}%
  \sbox2{\{}%
  \ifdim\ht0=\ht2
    \{\kern-.45\wd2 \{#2\}\kern-.45\wd2 \}%
  \else
    \mathopen{#1\{\kern-.5\wd0 #1\{}
    #2
    \mathclose{#1\}\kern-.5\wd0 #1\}}
  \fi
}

\usepackage[T1]{fontenc}
\usepackage{textcomp}
\usepackage[scaled]{helvet}

\newcommand{\bb}[1]{\mathbf{#1}}

\newcommand{\bk}{{\bf k}}

\newcommand{\bn}{{\bf n}}
\newcommand{\bK}{{\bf K}}

\newcommand{\bv}{{\bf v}}

\newcommand{\bx}{{\bf x}}

\newcommand{\LW}{{\mathcal{L}^W}}

\newcommand{\kpar}{{k_{\parallel}}}

\newcommand{\I}{\mathrm{i}}

\usepackage{mathrsfs}

\title{Unfitted Nitsche's method for computing wave modes in topological materials}
\author{Hailong Guo\thanks{School of Mathematics and Statistics,  The University of Melbourne,  Parkville, VIC 3010, Australia   (hailong.guo@unimelb.edu.au).}
\and %
Xu Yang\thanks{Department of Mathematics, University of California, Santa Barbara, CA, 93106, USA (xuyang@math.ucsb.edu).}
\and%
Yi Zhu\thanks{Yau Mathematical Sciences Center and Department of Mathematical Sciences, Tsinghua University, Beijing, 100084, People's Republic of China (yizhu@mail.tsinghua.edu.cn).}
}

\begin{document}

\maketitle

%\medskip
%
%\begin{center}
%
%\end{center}
\medskip

\begin{abstract}
In this paper, we propose an unfitted Nitsche's method for computing wave modes in topological materials. The proposed method is based on  the Nitsche's technique to study the performance-enhanced topological materials which have strongly heterogeneous structures ({\it e.g.}, the refractive index is piecewise constant with high contrasts). For periodic bulk materials, we use Floquet-Bloch theory and solve an eigenvalue problem in a torus with unfitted meshes. For the materials with a line defect, a sufficiently large domain with zero boundary conditions is used to compute the localized eigenfunctions corresponding to the edge modes. The interfaces are handled by  the  Nitsche's method on an unfitted uniform mesh. We prove the proposed methods converge optimally. Several  numerical examples are presented  to validate the theoretical results and demonstrate the capability of simulating topological materials.

%\textbf{YZ: Shall we mention the analysis results?}

\vskip .3cm
{\bf AMS subject classifications.} \ {}%{Primary 65N50, 65N30; Secondary 65N15, 53C99}
\vskip .3cm

\end{abstract}

\section{Introduction}
The past decade has witnessed an explosion of research on topological materials. The delicate structures of these materials admit novel and subtle propagating wave patterns which are immune to backscattering from disorder and defects\cite{rechtsman2013photonic,lu2014topological,plotnik2013observation,Shvets-PTI:13,Ablowitz_Zhu_HC_12}. The underlying mechanism is the existence of so-called ``topologically protected edge states''. These wave modes, which propagate along and decay rapidly transverse to the edge, are robust against local defects. Thus they can be used to transfer energy, information and so on. Over the past few years, in addition to the electronic system in which the topological phenomena was firstly studied, such topological phenomena have been experimentally realized in many other physical systems, such as electromagnetic waves in photonic systems and  acoustic waves in phononic systems \cite{Sheng:15,lu2014topological,MKW:15,susstrunk2015observation,ablowitz2009conical, WZLS2019}. 

%To study topological phenomena in a physical system, the key is to generate these edge states. However, this is not an easy task from the analytical aspect. So far, analytical formulae have been obtained in some simple discrete tight-binding models (\refs). For continuous models, the analytical structure can only be investigated in a very narrow parameter regime (\refs). Generally, we use numerical approaches to generate the edge states in most real circumstances despite these analytical progress.

There are many physical models which admit topologically protected edge states. This work is concerned with wave modes in topological photonic materials. The  mathematical problem that we study is the following eigenvalue problem
\begin{equation}\label{eq:eigen}
\LW \Psi \equiv -\nabla\cdot W(\bx)\nabla  \Psi(\bx) = E \Psi(\bx), \quad \bx=(x_1,x_2) \in \mathbb{R}^2.
\end{equation}
This equation can arise in the in-plane propagation of electromagnetic waves in a photonic crystal whose permittivity is invariant along the longitudinal direction. In this scenario,the electromagnetic fields $(E_1, E_2, E_3, H_1,H_2,H_3)$ can be divided into two decoupled components: transverse electric (TE) mode $(E_1, E_2, H_3)$ and transverse magnetic (TM) mode $(H_1, H_2, E_3)$. The  addimissible  TE modes in a specific material, characterized by the material weight function $W(\bx)$, satisfy the above eigenvalue problem \eqref{eq:eigen}. Here, $\Psi(\bx)$ corresponds to the longitudinal magnetic field $H_3$ and  the eigenvalue $E$ equals $\omega^2$ with $\omega$ being the frequency of the electromagnetic fields.
The other two components of the TE modes are $(E_1, E_2)=\pm\frac{i}{\sqrt{E}}\left(-\partial_{x_2}\Psi, \partial_{x_1}\Psi\right)$ corresponding to frequency $\pm \sqrt{E}$ respectively. We refer to \cite{LWZ2019,hu2020linear} for more details. Though the eigenvalue problem \eqref{eq:eigen} can also be obtained in other physical systems such as acoustic waves, we restrict our physical applications in the photonic aspect.

To ensure the existence the topological edge states, delicate structures are required for the material weight $W(\bx)$. Here we focus on the honeycomb-based material weight. The corresponding material is referred to as  ``photonic graphene''. Specifically, the material weight is of the form
\begin{equation}\label{eq:Wdef}
W(\bx)=A(\bx)+\delta\kappa(\delta \bk_2\cdot \bx)B(\bx),
\end{equation}
where $A(\bx)$ and $B(\bx)$ are hexagonally periodic Hermitian matrices, $\kappa(\cdot)\in \mathbb{R}$ is a bounded transition function,  $\delta>0$ is a parameter characterizing the intensity and width of the transition, and the detailed conditions are given in Section 2. From the application point of view, we need to obtain the bulk property ({\it i.e.}, $\delta=0$ ) and edge state property. Understanding the bulk property requires that we solve the eigenvalue problem in a torus using the Floquet-Bloch theory. A topological material can be constructed by gluing two bulks together by the transition function. Consequently, to investigate the wave modes in topological materials, we have to solve the eigenvalue problem in a cylinder since the existence of the transition breaks the periodicity along one direction.

Regarding the analytical understanding of the eigenvalue problem \eqref{eq:eigen} with the material weight \eqref{eq:Wdef}, Lee-Thorp {\it et al.} proved that the perfect honeycomb material weight ensures the existence of Dirac points in the spectrum, which can be used to construct topological edge states \cite{LWZ2019}. They also perturbatively constructed the edge states for specific parallel wavenumbers when $\delta$ is small and the material weight $W(\bx)$ is smooth. Their work greatly extends our knowledge on the understanding of topological edge states in a photonic system. However, their results are mostly on the existence aspect and lack the global structure of the bulk dispersion relation and edge states. All of these important studies rely on numerical simulations.
%
%
% Lee-Thorp, Weinstein and Zhu proved the existence of edge states with very specific wave lengthes when $\delta$ is very small and the material weight $W(\bx)$ is smooth. They also constructed a multi-scale formula of the edge states (\refs) based on the Bloch modes. Despite of the theoretical significance of the analysis, their results on the edge states have very limited applications when investigating a real photonic material. {\bf {\bf Rephase later}First, not smooth but piece-wise constant. Second, $\delta$ is not usually small. Third, the Bloch modes in the analytical formula have no explicit expression and can only be obtained numerically generally. Moreover, to obtain the global structure of the edge states and corresponding eigenvalues, we have to use numerical approaches. } Therefore, numerical approaches are highly desired for the study of the edge states.

Due to the particular structure of the photonic crystals, spectral method and finite element method are the two most popular methods. The spectral method utilizes the periodicity of the coefficients and eigenfunctions. Expanding the coefficients and Bloch modes into Fourier series and truncating the series into finite terms, the spectral method can achieve an exponential accuracy for smooth material weight $W(\bx)$. It is widely used for computing Bloch modes and corresponding energy surfaces \cite{SepctralBook,skorobogatiy2009fundamentals}.  However, one of the main shortcomings is that the final matrix is usually not sparse. If the material weight varies drastically or contain discontinuities which is the scenario of this paper , this method requires a large number of Fourier modes to resolve the coefficients. Thus the computation becomes expensive to solve the numerical algebra associated with a large and dense matrix. Actually, our numerical examples show that Fourier spectral method can lead to unreliable and even wrong results when the contrast (jump ratio) of the material weight is very high (see Figure \ref{fig:c30} ).

A competitive alternative is the finite element method.  In fact, finite element methods have been adopted in the computation of topological edge modes.   In our recent work \cite{GYZ2019},  we proposed a superconvergent post-processing method to compute topological edge modes for photonic graphene with smooth weight coefficient. The key idea is to recover more accurate gradients for numerical eigenfunctions and  use them to  improve the accuracy of   approximate  eigenvalues by using the Rayleigh quotient.  The superconvergent recovered gradient also enables us to reconstruct the full electromagnetic fields in real applications.   Due to the high contrast nature of the material coefficient, the proposed method can not be generalized directly.   The main difficulties are caused by the heterogeneous structure.    The existence of the jump in the material weight $W$ implies the non-smoothness of eigenfunctions across the material interfaces.  Although the classical finite element methods will work if the underlying mesh is fitted to the interface \cite{CZ1998, Ba1970, GY20181}, it is in general time-consuming and nontrivial to generate a body-fitted mesh.  The drawbacks become more serious for the interface with complicated geometric structure.  For the honeycomb structure,  the discontinuities in the material weight function is copied periodically,  which makes the generation of body-fitted meshes become challenging. Furthermore, the unstructured nature of the body-fitted meshes will introduce additional difficulties to impose the periodic or Bloch periodic boundary condition.
Those difficulties can be alleviated by adopting the unfitted numerical methods where the underlying meshes are independent of the location of the material interface.
To handle the non-smoothness across the material over the interface, one may need modify the finite difference stencil\cite{Pe1977, Pe2002, LL1994, LI2006},  finite element basis functions \cite{Li1998, LLW2003, HSWZ2013, HL2005, LLZ2015,  GY20172, GYZ2018}, or the weak formulation \cite{HH2002, BCHLM2015,AHD2012, GY20183}.

The main purpose of the paper is to propose a new kind of unfitted Nitsche's method based on the Floquet-Bloch transformation for computing  the  dispersion relation and wave modes in a honeycomb structure with strong heterogeneities.  The unfitted Nitsche's method was originally proposed in \cite{HH2002}  for the elliptic interface problem with real coefficients.  The key idea is to construct the approximation on each fictitious domain induced by the material interface and couple them together by the  Nitsche's technique \cite{Ni1971}.  For the development and application of the unfitted Nitsche's method, the interesting readers are referred to the recent review paper \cite{BCHLM2015}. Compared to the existing unfitted Nitsche's methods \cite{HH2002, BCHLM2015, AHD2012}, the proposed unfitted Nitsche's method uses Floquet-Bloch theory and solves an eigenvalue problem in a torus. For the $\mathcal{C}$-symmetry breaking case where the eigenvalue problem contains complex matrix-valued coefficients, a sufficiently large domain with zero boundary conditions is used to compute the localized eigenfunctions (edge mode).

%To deal with the Bloch periodic boundary condition,  we need to apply the Floquet-Bloch transform (\refs) to turn the interface eigenvalue problems into an equivalent to interface
%eigenvalue problem involving ({\color{red} it looks very hard for me to a correct word to describe it}).     The periodic boundary condition can be easily imposed on the uniform meshes.   ({\color{red} I am still organizing my language to state the difference between the proposed unfitted Nitsche's method and existing method. In particular, I want to emphases that the existing  method works only  for the diffusion equation with real-valued coefficient but here we have complex-valued matrix coefficient.})

One of the difficulties in analyzing the stability of the discrete Nitsche's bilinear form is that it involves the solution itself in addition to its gradient.  To the best of our knowledge, the existing unfitted Nitsche's method only focuses on the pure diffusion equation.  To establish the stability, we need the trace theorem on cut elements, i.e. elements cut by the interface.   The existing trace theorem  \cite{BCHLM2015, HH2002}  for the cut element involves both parts of the cut element. Direct application of the theorem is not able to entitle us the full possibility to prove the coercivity of the Nitsche's bilinear form.    Therefore, we build up a new trace inequality which involves only one part of the cut element.   The new trace inequality enables us to establish the stability and continuity for  Nitsche's bilinear form in term of the energy norm.  Using the approximation theory of the compact operator \cite{BO1991} and the interpolation error estimates, we are able to show the optimal convergence results for both discrete  eigenvalue and eigenfunctions  using   the proposed unfitted Nitsche's method.  In particular, the established error  estimates are independent of the location of the interface.    Furthermore, we show that there is no pollution in the numerical spectrum.

The rest of the paper is organized as follows. In Section~\ref{sec:phys}, we present the physical background of photonic graphene and the mathematical setup. In Section~\ref{sec:disp}, we focus on the computation of the  dispersion relation and  wave modes.   We start the section by  introducing the formulation of the unfitted Nitsche's method in the torus which gives us the unperturbed bulk properties.  The stability and continuity of the unfitted Nitsche's weak formulation are established.  Then, we extend the unfitted Nitsche's method to compute wave modes in a cylinder domain which corresponds to the physical setup of topological materials.   In Section~\ref{sec:symbreak}, we prove the numerically approximated eigenpairs converge optimally to the exact eigenpairs.   In Section~\ref{sec:num}, we present several numerical examples to justify the theoretical results.  We make conclusive remarks in Section~\ref{sec:conclusion}.
%{\bf plz introduce the standard FEM.  Shortcomings: low accuracy and the accuracy inconsistence between the eigenfunctions and their gradients which togenther consist the electromagnetic fields in real application. Then lead to gradient recovery. }
%
%{\bf more about gradient recovery and our approaches, explain why different from the standard gradient recovery}
%
%
%{\bf organization}

% {\bf Some repeated explanations. Some parts in two sections should merge into one}

\section{Physical problems and preliminaries}\label{sec:phys}

We will focus on the honeycomb-based photonic materials, and present the physical setup and briefly review the underlying theory.
\subsection{honeycomb structured material weight}
We consider the following specific hexagonal lattice
\begin{equation}
\Lambda=\mathbb{Z}\bv_1+\mathbb{Z}\bv_2=\left\{m_1\bv_1+m_2\bv_2: m_1,\  m_2\in\mathbb{Z}\right\},
\end{equation}
with the lattice basis vectors
\[
 \bv_1=\begin{pmatrix} \frac{\sqrt{3}}{2} \\  \\ \frac{1}{2}\end{pmatrix},\quad
 \bv_2=\begin{pmatrix}\frac{\sqrt{3}}{2} \\ \\ -\frac{1}{2} \end{pmatrix}.\]
The fundamental cell is chosen to be the parallelogram:
\begin{eqnarray}\label{cell}
\Omega = \{\theta_1\bv_1+\theta_2\bv_2: 0\le\theta_j\le 1, j=1, 2\},
\end{eqnarray}
with $|\Omega|$ standing for the area of $\Omega$.

The dual lattice
\begin{equation}
	\Lambda^* = \{m_1\bk_1+m_2\bk_2: \bm=(m_1, m_2)\in\mathbb{Z}^2\}=\mathbb{Z}\bk_1\oplus\mathbb{Z}\bk_2,
\end{equation}
is generated by the dual lattice vectors $\bk_1,~\bk_2$ which satisfy $\bk_i\cdot\bv_j = 2\pi\delta_{ij},~(i, j=1,2)$.
Specifically, the dual lattice vectors are
\begin{equation}
  \bk_1=\frac{4\sqrt{3}}{3}
    \begin{pmatrix}
      \frac12\\
      \frac{\sqrt{3}}{2}
    \end{pmatrix}
    ,\quad\bk_2=\frac{4\sqrt{3}}{3}
    \begin{pmatrix}
      \frac12\\
      -\frac{\sqrt{3}}{2}
    \end{pmatrix}.
\end{equation}
Throughout this work, we choose the parallelogram $\Omega^*$:
\begin{equation}\label{dualcell}
\Omega^* = \{\theta_1\bk_1+\theta_2\bk_2: -\frac12\le\theta_j\le \frac12, j=1, 2\},
\end{equation}
as the fundamental dual cell.

Let $\mathbf{A}=\frac{1}{3}(\bb v_1+\bb v_2)\in \Omega, \mathbf{B}=\frac{2}{3}(\bb v_1+\bb v_2)\in \Omega$. Define the honeycomb lattice $\Lambda_h=\left(\bb A+\Lambda\right) \bigcup\left( \bb B+\Lambda\right)$. Note that $\Lambda_h$ has two sites per unit cell.

Let $B_r(\bx_0)$ be the ball centered at $\bx_0$ with the radius $r$. Throughout this work, we require that $r<\frac{1}{2}|\bb A-\bb B|$ which implies that $B_r(\bb A)$ and $B_r(\bb B)$ are disjoint. We divide the fundamental cell into two parts, $\Omega_1=B_r(\bb A)\bigcup B_r(\bb B)$ and $\Omega_2=\Omega/\Omega_1$.  The interface $\Gamma$ on the fundamental cell  is defined as the intersection  of $\Omega_1$ and $\Omega_2$, i.e. $\Gamma = \Gamma_1\cap \Gamma_2$. 
Define the piece-wise honeycomb function 
\begin{equation}
   \epsilon(\bx)=\left\{\begin{array}{l}
  \epsilon_{\bb A}, \quad \text{if} \ \bx \in B_r(\bb A)+\Lambda, \\
  \epsilon_{\bb B}, \quad \text{if} \ \bx \in B_r(\bb B)+\Lambda,\\
   \epsilon_0, \,\quad \text{if} \ \bb x \in \Omega_2+\Lambda,
    \end{array}\right.
\end{equation}
where $\epsilon_j, ~ j=\bb A, \bb B,0$, are positive constants. $\epsilon_0$ is regarded as the value of the background and $\epsilon_{\bb A}, ~\epsilon_{\bb B}$ are the values against the background.  It is obvious that $\epsilon(\bx)$ is $\Lambda$-periodic, i.e., $\epsilon(\bx+\bv)=\epsilon(\bx)$ for all $\bv \in \Lambda$.

In this work, we use the following material weight as our prototype
\begin{equation}\label{eq:materialweight}
 W(\bx)=\begin{pmatrix}
                     \epsilon(\bx) & i\gamma \\
                     -i\gamma & \epsilon(\bx) \\
                   \end{pmatrix}^{-1}.
\end{equation}
This material weight corresponds to the magneto-optical material \cite{HR:07}. $\gamma\in \mathbb{R}$ is called Farady-rotation constant satisfying $\min(\epsilon(\bx)^2-\gamma^2)>c_0>0$, which ensures $ W(\bx)$ is uniformly elliptic. In real materials,  the strength of the Faraday-rotation is much smaller than the permitivity $\epsilon$, hence
\begin{equation}\label{eq:materialweight2}
W(\bx)\approx \epsilon(\bx)^{-1}I+\gamma\epsilon^{-2}\sigma_2,
\end{equation}
where $\sigma_2=\begin{pmatrix}0 & -i \\ i & 0 \end{pmatrix}$ is a Pauli matrix.

\subsection{Eigenvalue problem in a torus}
Consider the material weight of the form \eqref{eq:materialweight} or \eqref{eq:materialweight2}. $W(\bx)$ is $\Lambda$-periodic when $\gamma$ is constant. We can restrict our analysis in a torus by Floquet-Bloch theory. Before proceeding further, we introduce the following function space
\begin{eqnarray*}
&&L^2_{per}(\Lambda)=\left\{f(\bx)\in L^2_{loc}\left(\mathbb{R}^2,\mathbb{C}\right):\ f(\bx+\bv)=f(\bx), \forall \bv \in \Lambda, \bx \in \mathbb{R}^2 \right\}\\
&&L^2_{\bk}(\Lambda)=\left\{g(\bx):\  e^{-i\bk\cdot \bx}g(\bx)\in L^2_{per}(\Lambda) \right\}.
\end{eqnarray*}
Note that functions in $L^2_{\bk}(\Lambda)$ are quasi-periodic. Namely, if $g(\bx) \in L^2_{\bk}(\Lambda)$, then $g(\bx+\bv)=e^{i\bk\cdot\bv}g(\bx), \forall \bv \in \Lambda$.
Similarly, we can also define $H_{per}^s(\Lambda)$ and $H_{\bk}(\Lambda)$ in a standard way.

%In photonics, one of the main tasks is to study the adimissible electromagnetic modes.
%
%
%\footnote{\bf{YZ}, I don't want to lose the generality of our work and move the model reduction to appendix and claim that the physical setup there is just a special example. We can actually apply our method to other physical setups}An in-plane propagating electromagnetic mode can be divided into two decoupled components, TE mode $(E_1, E_2, H_3)$ and TM mode $(H_1, H_2, E_3)$, (see appendix for details). The longitudinal magnetic field $H_3$ satisfies the following eigenvalue problem after rescaling
%\begin{equation}\label{eq:eigen2}
%-\nabla\cdot W(\bx)\nabla  \Psi(\bx) = E \Psi(\bx), \quad \bx=(x_1,x_2) \in \mathbb{R}^2,
%\end{equation}
%and  longitudinal electric field $E_3$ statisfies
%\begin{equation}\label{eq:eigen3}
%-\Delta  \Psi(\bx) = E w(\bx) \Psi(\bx), \quad \bx=(x_1,x_2) \in \mathbb{R}^2,
%\end{equation}
%where $W(\bx)$ and $w(\bx)$ are material weights which are related to the permittivity and permeability. This work focus on the TE mode supported by honeycomb-based media though our method and result are also applicable to TM mode with some considerable modifications.
%
%%\subsection{Periodic material weight and Floquet-Bloch theory}
According to Floquet-Bloch theory, the spectrum of $\mathcal{L}^W$ in $L^2(\mathbb R^2)$ can be represented by the spectrum $\mathcal{L}^W$ in $L_{\bk}^2(\Lambda)$. Namely, we solve the following $L^2_\bb{k}(\Lambda)$-eigenvalue problem
\begin{equation}\label{eq:eigenLk}
\mathcal{L}^W \Phi(\bx) = E \Phi(\bx),\quad \quad \Phi(\bx) \in L^2_{\bb k}(\Lambda).
\end{equation}
Due to the periodicity, we can restrict $\bb k$ in the fundamental dual cell $\Omega^*$. For a fixed $\bb k\in \Omega^*$, there exists a sequence of pairs $\left(E_m(\bb k), \Phi_m(\bx;\bb k)\right), m=1, 2,\cdots$ satisfying the above eigenvalue problem. Here $E_m(\bb k),  ~m=1,2,\cdots$ are called dispersion band functions which have been ordered as $0<E_1(\bb k)\le E_2(\bb k)\le E_3(\bb k)\le \cdots$. The corresponding eigenfunctions $\Phi_m(\bx;\bb k)$ are referred  to as the Bloch waves. Moreover, the set $\left\{\Phi_m(\bb x;\bb k),~m\in \mathbb N,~\bb k\in \Omega^*\right\}$ forms a ``generalized'' basis of $L^2(\mathbb R^2)$ and the spectrum of $\mathcal{L}^W$ in $L^2(\mathbb{R}^2)$, $\sigma(\mathcal{L}^W)$, coincides with the Bloch spectrum, the union of the images of all the mappings $E_m(\bb k)$, i.e.,
\begin{equation}
\sigma(\mathcal{L}^W)= \bigcup_{m=1}^{\infty}\left[\inf_{\bb k \in\Omega^*} E_m(\bb k),\sup_{\bb k \in\Omega^*} E_m(\bb k)\right].
\end{equation}

%Alternatively, setting $\Phi_m(\bx;\bb k)=e^{i\bk\cdot \bx}\phi_m(\bx;\bb k)$, we translate the eigenvalue problem \eqref{eq:eigenLk} to the following $L^2_{per}(\Lambda)$-eigenvalue problem \begin{equation}\label{eq:eigenLp}
%\mathcal{L}^W(\bk) \phi(\bx) = E \phi(\bx),\quad \quad \phi(\bx) \in L^2_{per}(\Lambda).
%\end{equation}
% where
% \begin{equation}
% \mathcal{L}^W(\bk)=(\nabla+i\bk)\cdot W(\bx)(\nabla+i\bk).
% \end{equation}

In general, it is impossible to solve the eigenvalue problem \eqref{eq:eigenLk} analytically. A natural numerical scheme is the spectral method.  Namely, we can expand $W(\bx)$ and $\Phi(\bx)$ into their Fourier series. By truncating the series into finite terms, we can easily solve the reduced eigenvalue problem for a matrix. If $W(\bx)$ is smooth,  $ \Phi(\bx)$ is also smooth. We only need a few terms to approximate $W(\bx)$ and $\phi(\bx)$ due to the exponential accuracy. The shortcoming of this method is that the resulting matrix is not sparse. When we need a large number of terms to approximate $W(\bx)$, this method becomes costly and sometimes lead to wrong results. A typical scenario is that $W(\bx)$ changes greatly or is even discontinuous and this regime is exactly what we are going to handle.

%{\bf Add our method}
%In this work, we are interested in the material weight
%\begin{equation}\label{materialweight}
% W(\bx)=\begin{pmatrix}
%                     \epsilon(\bx) & i\gamma \\
%                     -i\gamma & \epsilon(\bx) \\
%                   \end{pmatrix}^{-1}.
%\end{equation}
%Here $\gamma\in \mathbb{R}$ is called Faradi-rotation constant satisfying $\min(\epsilon(\bx)^2-\gamma^2)>c_0>0$, which ensures $\bb W(\bx)$ is uniformly elliptic.

If $\gamma=0$, $W(\bx)$ is a honeycomb structured material defined in \cite{LWZ2019}, i.e., $W(\bx)$ is even, real and $\frac{2\pi}{3}$-rotation invariant. According to \cite{LWZ2019}, there generically exist the so-called Dirac points--conical singularities in the dispersion band functions $E_m(\bk)$ at $\bK=\frac{1}{3}(\bk_1-\bk_2), \bK'=-\bK$ for some $m$. If $\gamma\neq 0$ but is still a constant, the material weight $W(\bx)$ is now complex, local spectral gaps open near the Dirac points due to the complex-conjugate symmetry breaking.  %In Fig. xx, we use our method to demonstrate the existence of Dirac points and gap opening respectively.

\subsection{Honeycomb structured material weight with a line defect} Dirac points provide a mechanism to generate the so-called topological edge states via introducing a line defect. Define a transition function (referred as to domain wall function) $\kappa(\zeta) \in L^\infty(\mathbb{R},\mathbb{R})$ with $\kappa(\pm\infty)=\pm \kappa_\infty$. Without loss of generality, we require that $\kappa_\infty>0$. A typical example of this transition function is the step function
\begin{equation}\label{stepfun}
  \kappa(\zeta)=\left\{\begin{array}{ll}-\kappa_\infty, & \zeta<0,\\ 0,& \zeta=0,\\ +\kappa_\infty,& \zeta>0.\end{array}\right.
\end{equation}
Its smooth counterpart is $\kappa(\zeta)=\kappa_\infty\tanh(\zeta)$.

A line defected is introduced if we choose the Faradi-rotation $\gamma$ in \eqref{eq:materialweight} to be a transition function along a direction. Namely $\gamma=\kappa(\mathbf{n}\cdot \bx)$ where $\mathbf{n}\neq 0$ is the normal direction of the line defect. Obviously, if $\kappa(\zeta)$ is the step function \eqref{stepfun}, the line $\mathbf{n}^\perp\mathbb{R}$ is the interface of two different materials (we also call it an edge). In this work we take Zigzag edge as our prototype. In this case, $\mathbf{n}=\bk_2$ and the line $\mathbf{v}_1\mathbb{R}$ is the edge. Note that $W(\bx)$ is periodic along $\bv_1$ direction but loses the periodicity along $\bv_2$ direction.

Let $\Sigma = \mathbb{R}^2/\mathbb{Z}\bv_1$ be a cylinder.  The fundamental domain for $\Sigma$ is
$\Omega_{\Sigma} \equiv \{ \tau_1\bv_1+\tau_2\bv_2: 0 \le \tau_1 \le 1, \tau_2 \in \mathbb{R}\}$. Define the function spaces
\begin{eqnarray*}
&&L^2_{per}(\Sigma)=\left\{f(\bx) \in L^2\left(\Omega_{\Sigma},\mathbb{C}\right):\ f(\bx+\bv_1)=f(\bx)\right\}\\
&&L^2_{k_\parallel}(\Sigma)=\left\{g(\bx)\in L^2\left(\Omega_{\Sigma},\mathbb{C}\right):\  g(\bx+\bv_1)=e^{ik_\parallel}g(\bx) \right\}.
\end{eqnarray*}

%The spectrum of the $\mathcal{L}^W$ in $L^2(\mathbb{R}^2)$ can be decomposed in to $L^2_{k_\parallel}(\Sigma)$
We solve the eigenvalue problem \eqref{eq:eigen} in $L^2_{k_\parallel}(\Sigma)$. For a given $k_\parallel\in [0, 2\pi]$, the continuous spectrum of $\mathcal{L}^W$ can be obtained by letting $\bk_2\cdot \bx $ tend to $\pm \infty$. Indeed,
\begin{equation}
\sigma_{c, k_\parallel}(\mathcal{L}^W)=\bigcup_{m=1}^{\infty}\left[\inf_{\lambda \in[-1/2, 1/2]} E_m(\lambda \bk_2+\frac{1}{2\pi}k_\parallel\bk_1),\sup_{\lambda \in[-1/2, 1/2]} E_m(\lambda \bk_2+\frac{1}{2\pi}k_\parallel\bk_1)\right]
\end{equation}

Due to the subtle symmetries of the setup, there exists point spectrum and the corresponding eigenfunctions are referred as to edge states. Namely, $\Psi(\bx)$ satisfies
\begin{align}
&\mathcal{L}^W\Psi(\bx;\kpar) = E(\kpar)\Psi(\bx;\kpar), \label{equ:dw_evp}\\
&\Psi(\bx+\bv_1;\kpar)=e^{\I\kpar}\Psi(\bx;\kpar),\label{equ:pseudo-per}\\
&\Psi(\bx;\kpar) \to 0\ \ {\rm as}\ \ |\bx\cdot\bk_2|\to\infty.  \label{equ:localized} .
\end{align}

\section{Unfitted Nitsche's method}\label{sec:disp}  In this section,  we propose the Floquet-Bloch theory based unfitted Nitsche's methods for simulating topological materials. We first focus on the computing  the bulk dispersion relations.  Then, we extend the method to computing wave modes in topological materials.  In this paper, we use $C$, with or without a  subscript, to denote a generic constant,  which can be different at different occurrences. In addition, it is independent of the mesh size and the location of the interface.

\subsection{Unfitted Nitsche's method for computing dispersion relation} In this section, we are interested in the efficient numerical solution of the eigenvalue problem \eqref{eq:eigenLk} in a torus. One of the main numerical difficulties is the existence of the high contrast in the material weight $W(\bx)$, which may lower the regularity of eigenfunctions.  To model the discontinuity, we use the interface
conditions as \cite{LI2006} and the $L^2_\bb{k}(\Lambda)$-eigenvalue problem \eqref{eq:eigenLk}  can be converted into the following interface $L^2_\bb{k}(\Lambda)$-eigenvalue problem
\begin{align}\label{eq:iteigenLk}
&\mathcal{L}^W \Phi(\bx) = E \Phi(\bx), \\
& \left \llbracket \Psi\right \rrbracket  = \left \llbracket W\frac{\partial \Psi}{\partial n}\right \rrbracket = 0, \quad \text{on } \Gamma;
\label{equ:tb4}
\end{align}
where  $\Phi(\bx) \in L^2_{\bb k}(\Lambda)$,  $\left\llbracket  v\right \rrbracket$ is the jump in value of a function $v$  crossing the interface $\Gamma$, and $n$ is the unit outer normal vector of $\Gamma$.

To deal with quasi-periodicity of functions in $ L^2_{\bb k}(\Lambda)$, we apply the Floquet-Bloch transform $\Phi(\bx;\bb k)=e^{i\bk\cdot \bx}\phi(\bx;\bb k)$.
We transfer the eigenvalue problem \eqref{eq:iteigenLk}--\eqref{equ:tb4} to the following interface $L^2_{per}(\Lambda)$-eigenvalue problem
\begin{align}
& \mathcal{L}^W(\bk) \phi(\bx) = E(\bk) \phi(\bx),\label{eq:iteigenLp}\\
& \left \llbracket \phi\right \rrbracket  = \left \llbracket W(\nabla + \I\bk)\phi\cdot n\right \rrbracket = 0, \quad \text{on } \Gamma; \label{eq:incond}
\end{align}
 where $ \phi(\bx) \in L^2_{per}(\Lambda)$ and
 \begin{equation}
 \mathcal{L}^W(\bk)=(\nabla+\I\bk)\cdot W(\bx)(\nabla+\I\bk).
 \end{equation}

 To address the numerical challenge brought by the interface condition \eqref{eq:incond},   the most straightforward idea is to use finite element methods with body-fitted meshes \cite{Ba1970, CZ1998} to resolve the discontinuity. However,  this brings two new {\it difficulties}: (1) the body-fitted meshes, in general, are unstructured meshes on which it is difficult to impose the periodic boundary conditions;  (2) it is technically hard to generate body-fitted meshes, in special for topological materials with complicated geometric structures and huge number of interfaces.  In this paper, we avoid those two difficulties by introducing the unfitted Nitsche's methods \cite{HH2002, BCHLM2015}.

\subsubsection{Unfitted Nitsche's method in a torus}
 This subsection is devoted to the unfitted Nitsche's method for  the interface $L^2_{per}(\Lambda)$-eigenvalue problem \eqref{eq:iteigenLp}--\eqref{eq:incond}.  \emph{To avoid  the generation of  body-fitted meshes for complicate topological structure and simplify the imposing of periodical boundary condition, we partition  the fundamental cell $\Omega$ using uniform triangular meshes.} The uniform triangulation is obtained by dividing $\Omega$ into $N^2$ sub-rhombuses with mesh size $h = \frac{\|\bv_1\|_2}{N}$ and then splitting each sub-rhombs into two isosceles triangles.  In addition, we assume that $N$ is sufficiently  large such that the following assumption holds:

\begin{assumption}\label{ass:inter}
The interface $\Gamma$ intersects each  interface element boundary $\partial K$ exactly twice, and each open edge at most once.
\end{assumption}

 The elements of $\mathcal{T}_h$ can be categorized into two different classes: regular elements and interface elements.  An element $\tau$ is called an interface element if
 the interface $\Gamma$ passes through $K$.  The  set of all elements that intersect the interface $\Gamma$ is denoted by $\mathcal{T}_{\Gamma,h}$. Then, it is easy to see that
\begin{equation}
\mathcal{T}_{\Gamma,h} = \left\{ K\in \mathcal{T}_h:  \Gamma\cap \overline{K} \neq \emptyset \right\}.
\end{equation}
Denote the union of all such type elements by
\begin{equation}
\Omega_{\Gamma,h} = \bigcup\limits_{K\in \mathcal{T}_{\Gamma,h}}K,
\end{equation}
and the set of all elements covering subdomain $\Omega_i$ by
\begin{equation}
\mathcal{T}_{i,h} = \left\{ K\in \mathcal{T}_h:   \overline{\Omega_i}\cap \overline{K} \neq \emptyset\right\}, \quad i = 1, 2.
\end{equation}
Let
\begin{equation}
\Omega_{i,h} = \bigcup\limits_{K\in \mathcal{T}_{i,h}}K, \quad \omega_{i,h} = \bigcup\limits_{K\in \mathcal{T}_{i,h}\setminus \mathcal{T}_{\Gamma,h}}K,\quad i = 1, 2.
\end{equation}
Figure \ref{fig:mesh} gives an illustration of $\Omega_{i,h}$ and $\omega_{i,h}$.  We remark that
$\Omega_{1,h}$ and $\Omega_{2,h}$ overlap on $\Omega_{\Gamma,h}$, which is shown as the blue part in Figures \ref{fig:in_mesh} and \ref{fig:ex_mesh}.

\begin{figure}
   \centering
   \subcaptionbox{\label{fig:meshwhole}}
  {\includegraphics[width=0.325\textwidth]{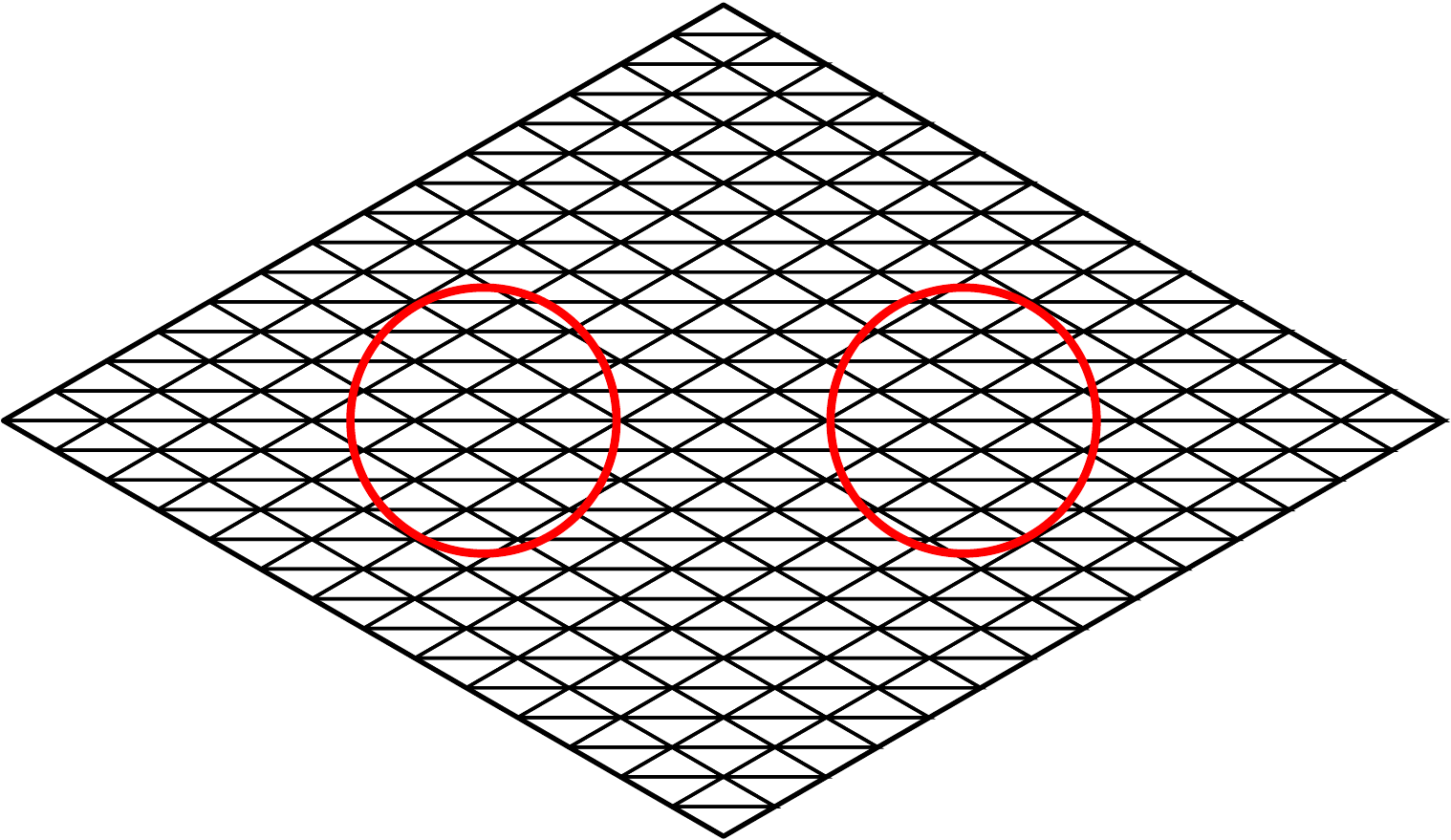}}
  \subcaptionbox{\label{fig:in_mesh}}
   {\includegraphics[width=0.325\textwidth]{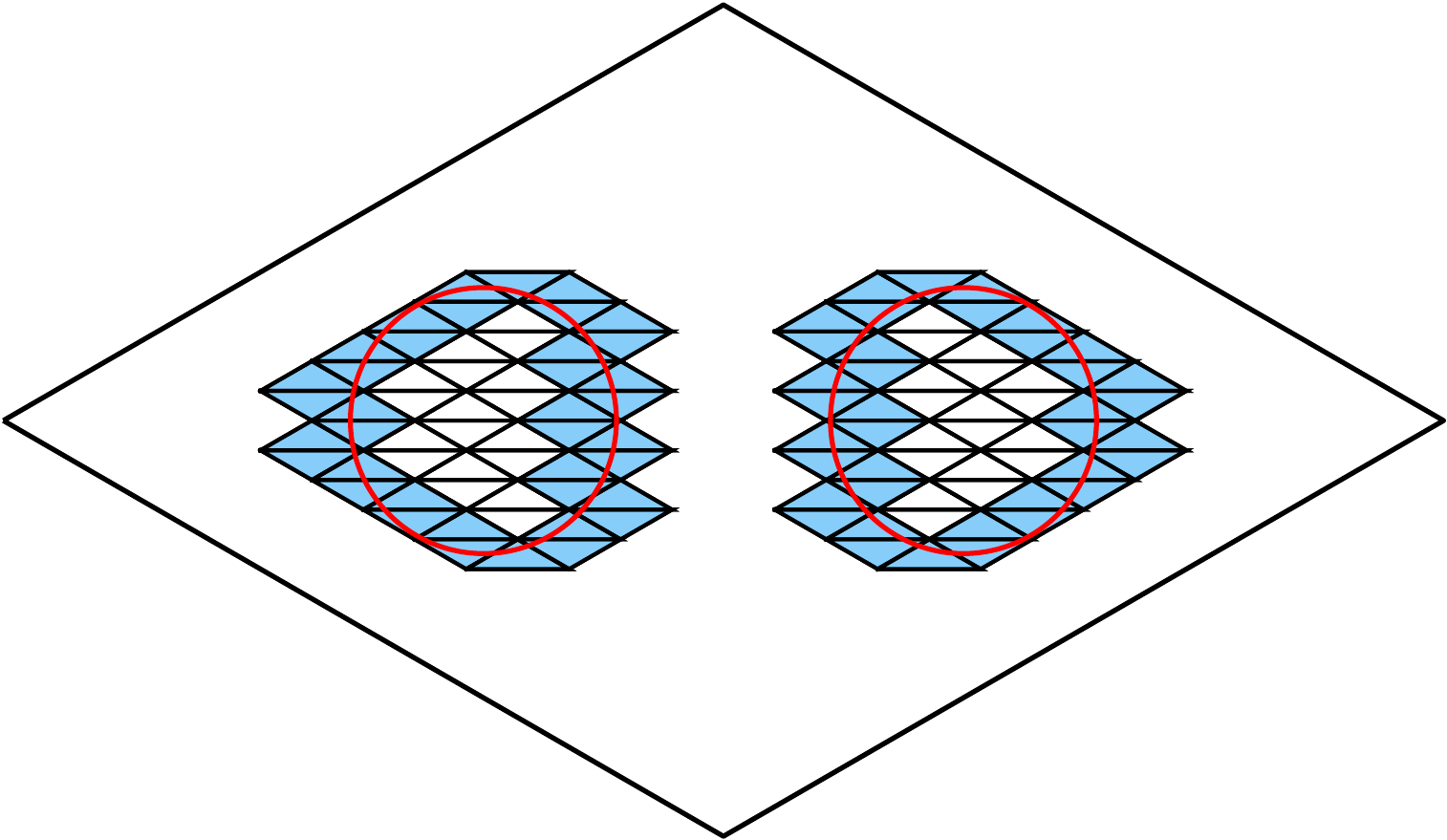}}
  \subcaptionbox{\label{fig:ex_mesh}}
  {\includegraphics[width=0.325\textwidth]{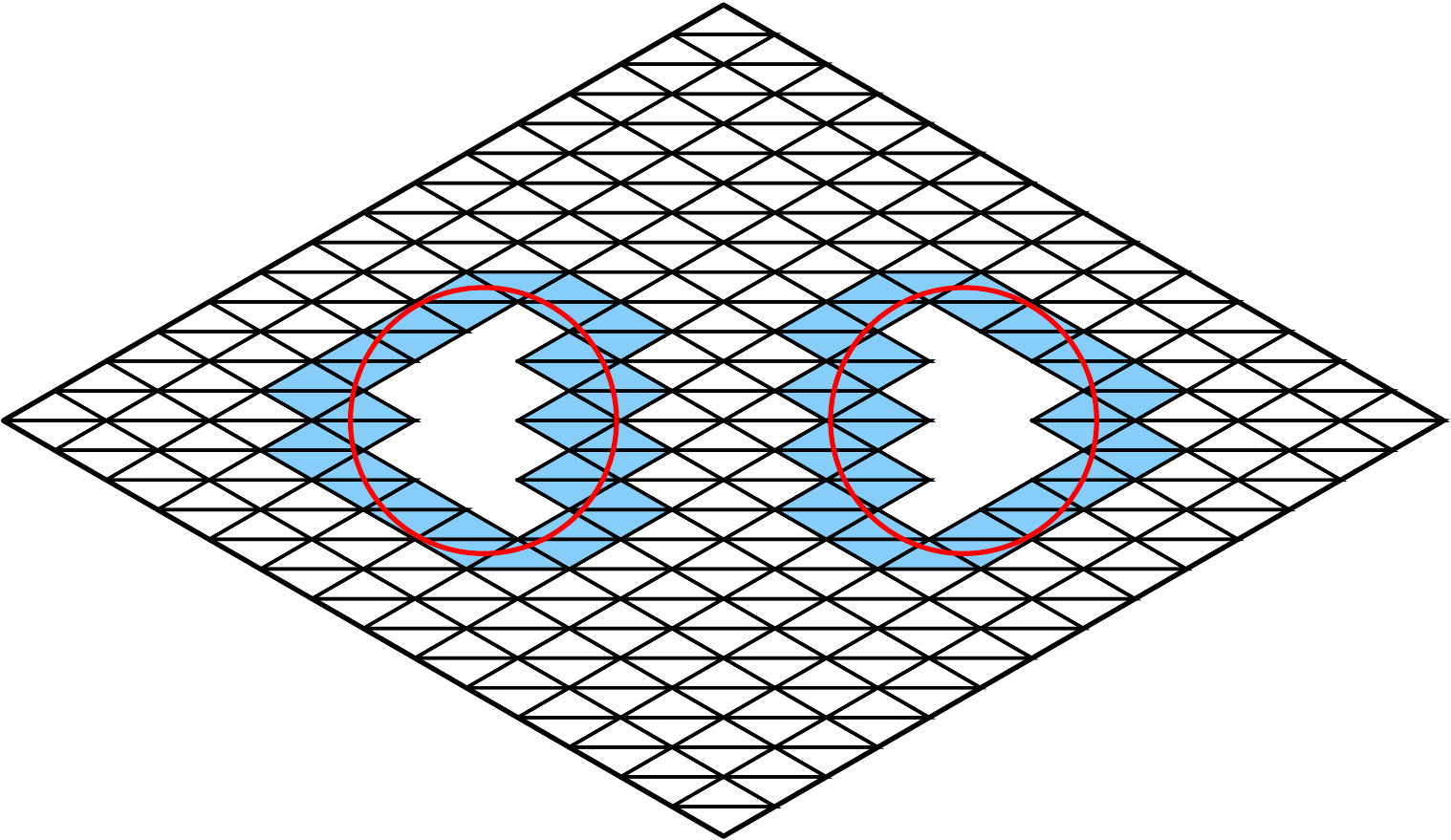}}
   \caption{Triangulation $\mathcal{T}_h$ on fundamental cell. (a): Triangulation $\mathcal{T}_h$; (b): Triangulation $\mathcal{T}_{1,h}$;
   (c): Triangulation $\mathcal{T}_{2,h}$.}\label{fig:mesh}
\end{figure}

One of  main ingredients of  the unfitted Nitsche's method is to define the finite element space  as
the direct sum of the standard  continuous linear  finite element space on $\Omega_{i,h}$.
For such a purpose, we let  $V_{i,h}$ be the standard continuous linear finite element space on $\Omega_{i,h}$, i.e.
\begin{equation}
V_{i,h} = \left\{ v \in C^0(\Omega_{i,h}):  v|_{K} \in \mathbb{P}_1(K) \text{ for any } K \in \mathcal{T}_{i,h}\right\},\quad i = 1, 2,
\end{equation}
where $\mathbb{P}_k(K)$ is the space of polynomials with  degree less than or equal to  $k$ on  the element $K$.
The finite element space  for  the unfitted Nitsche's method is defined as $V_h = V_{1,h} \oplus V_{2,h}$. In other words,
 \begin{equation}
V_h = \left\{ v_h = (v_{1,h}, v_{2,h}): v_{i,h}\in V_{i,h}, \, i = 1,2 \right\}.
\end{equation}
To impose the periodic boundary condition,  we introduce $V_{h, per}$ as a subspace of $V_h$ which is defined as
\begin{equation}
V_{h,per} = \left\{ v_h\in V_h: v_h(\bx+\bv)=v_h(\bx), \forall \bv \in \Lambda, \bx \in \mathbb{R}^2 \right\}.
\end{equation}
Note that a function in $V_h$ (or $V_{h, per}$)  is a vector-valued function from $\mathbb{R}^2\mapsto\mathbb{R}^2$, which has a zero component in $\omega_{1,h}\bigcup\omega_{2,h}$ but in general two non-zero components in $\mathcal{T}_{\Gamma,h}$.  It means that there are two sets
of basis functions for any element $K$ in $\mathcal{T}_{\Gamma,h}$: one for $V_{1,h}$ and the other for $V_{2,h}$.

For any interface element $K$ in $\mathcal{T}_{\Gamma,h}$,  let  $K_i = K \cap \Omega_i$ be  the part of $K$ in $\Omega_i$, where
$|K_i|$ is the area of $K_i$. Similarly, let  $\Gamma_K = \Gamma\cap K$ be the part of $\Gamma$ in $K$, where  $|\Gamma_K|$ is
the measure of $\Gamma_K$ in $\mathbb{R}^1$.
Different from the interface elliptic problem considered for  unfitted Nitsche's method in \cite{HH2002, GY20183},
the material weight coefficient $W(\bx)$ is complex and matrix-valued.   To increase the robustness of the Nitsche's method,
 we  introduce two weights using the maximal norm of $W$  inspired by   \cite{AHD2012}
\begin{equation}\label{eq:kapdef}
\kappa_1|_K = \frac{\|W_2\|_{\infty} |K_1|}{\|W_2\|_{\infty}|K_1|+\|W_1\|_{\infty}|K_2|},
\quad \kappa_2|_K = \frac{\|W_1\|_{\infty}|K_2|}{\|W_2\|_{\infty}|K_1|+\|W_1\|_{\infty}|K_2|},
\end{equation}
which satisfies that $ \kappa_1 + \kappa_2=1$.
Then,  we define the weighted  averaging of a function $ v_h$ on the interface $\Gamma$ as
\begin{equation}\label{eq:wadef}
\dgal{v_h} = \kappa_1v_{1,h} +\kappa_2v_{2,h}.
%\quad  \dgal{v_h} ^{\ast}= \kappa_2v_{1,h} +\kappa_1v_{2,h}.
\end{equation}
Furthermore, we define the constant  $\lambda_K$ as
\begin{equation}\label{eq:lamdef}
\lambda_K = \frac{h\|W_1\|_{\infty}\|W_2\|_{\infty}|\Gamma_K|}{\|W_2\|_{\infty}|K_1| + \|W_1\|_{\infty}|K_2|}.
 \end{equation}
 Based on $\lambda_K$, we define element-wise parameter $\lambda$ as  $\lambda|_{K} = \hat{\lambda} \lambda_K$ for some large enough positive number $\hat{\lambda}$ (called stabilizing parameter).
It is easy to see  that $\lambda_K \le \|W\|_{\infty} := \max(\|W_1\|_{\infty}, \|W_2\|_{\infty})$.

The unfitted Nitsche's method for the interface $L^2_{per}(\Lambda)$-eigenvalue problem \eqref{eq:iteigenLp}--\eqref{eq:incond} is to find the
the eigenpair $(\phi_h, E_h(\bk))\in V_{h, per}\times \mathbb{R}$ with $\phi_h \neq 0$ such that
\begin{equation}\label{eq:fem}
 a_h(\phi_h,  q_h) = E_h(\bk)b(\phi_h,  q_h),\quad  \forall  q_h \in V_{h, per},
\end{equation}
where
\begin{equation}\label{eq:disbil}
\begin{split}
  a_h(\phi_h ,  q_h) = &\sum_{i=1}^2\int_{\Omega_i}W(\nabla +\I\bk) \phi_h\cdot \overline{(\nabla + \I\bk) q_h}d\bx - \int_{\Gamma}
   \dgal{    W(\nabla+\I\bk)\phi_h\cdot n }\llbracket \overline{ q_h} \rrbracket ds-\\
  &\int_{\Gamma} \dgal{  \overline{ W(\nabla+\I\bk) q_h\cdot n }}\llbracket \phi_h \rrbracket ds
 +\frac{1}{h} \int_{\Gamma} \lambda \llbracket \phi_h \rrbracket \llbracket \overline{ q_h} \rrbracket ds,
\end{split}
\end{equation}
and
\begin{equation}\label{eq:disbb}
  b(\phi_h ,  q_h) = \int_{\Omega}\phi_h\cdot \overline{ q_h}d\bx,
\end{equation}
where  $h$ is the mesh size.
The weak formulation \eqref{eq:fem} is called the Nitsche's weak formulation.

\subsubsection{Well-posedness of the unfitted Nitsche's method in a torus}
In this part,  we shall show the unfitted Nistche's method is well-posed.
We start by showing
the following consistency result:
\begin{lemma}
Let  $(\phi, E)$ be the eigenpair of the interface $L^2_{per}(\Lambda)$-eigenvalue problem \eqref{eq:iteigenLp}--\eqref{eq:incond}.
Then $(\phi, E)$ satisfies
\begin{equation}\label{eq:consist}
 a_h(\phi,  q) = E(\bk) b(\phi,  q),\quad  \forall  q \in H^1_{per},
\end{equation}
\vspace{-0.2in}
\end{lemma}
\begin{proof}
 For any $\phi, q\in H^1_{per}$, we notice that  $\llbracket \phi \rrbracket  = \llbracket q \rrbracket = 0$ and hence  $a_h(\cdot, \cdot)$ is reduced to  the standard bilinear formulation. Then \eqref{eq:consist} follows  by the Green's formula on each subdomain $\Omega_i$
 and the interface condition \eqref{eq:incond}.
\end{proof}

Thanks to the above Lemma, we can easily deduce following corollary which is known as the Galerkin  orthogonality:
\begin{corollary}
 Let  $(\phi, E(\bk))$ be the eigenpair of the interface $L^2_{per}(\Lambda)$-eigenvalue problem \eqref{eq:iteigenLp}--\eqref{eq:incond}
 and $(\phi_h, E_h(\bk))$ be the corresponding approximate   eigenpair  by the unfitted Nitsche's method. Then we have
 \begin{equation}\label{eq:gal}
 a_h(\phi-\phi_h,  q_h)  = 0, \quad \forall  q_h \in V_{h, per}.
\end{equation}
\vspace{-0.2in}
\end{corollary}

To analyze the stability  of the  bilinear form $a_h(\cdot, \cdot)$, we introduce the following mesh-dependent norm \cite{BCHLM2015,HH2002}
\begin{equation}\label{def:md_norm}
|||\phi|||_h^2 = \| (\nabla + \I\bk) \phi\|_{0, \Omega_1\cup \Omega_2}^2
+\sum\limits_{K\in \mathcal{T}_{\Gamma,h}}h^{-1}\|\llbracket \phi\rrbracket\|_{0,\Gamma_K}^2.
\end{equation}

We prepare our proof of  the stability of the  bilinear form  by establishing the following Lemma, whose  proof is given in Appendix A.
\begin{lemma}\label{lem:imtr}
Let $\phi_h$ be a finite element function in $V_h$. Then the following inequalities hold:
 \begin{align}
&\|\phi_{i,h}\|^2_{0, \Gamma_K}  \le C_1\frac{h^2|\Gamma_K|}{|K_i|} \|\nabla \phi_{i,h}\|^2_{0, K_i}, \label{eq:l2tr}\\
&   \|\nabla \phi_{i,h}\|^2_{0, \Gamma_K}\le  C_2\frac{|\Gamma_K|}{|K_i|} \|\nabla \phi_{i,h}\|^2_{0, K_i} \label{eq:h1tr}.
\end{align}
\vspace{-0.2in}
\end{lemma}
%\begin{proof}
% We prove this Lemma in Appendix A.
%\end{proof}

\begin{remark}
 The inequality of \eqref{eq:l2tr} is a refinement of the trace inequality on a cut element  in \cite{HH2002}.  It is the key to show the stability of the bilinear form.
\end{remark}

 Based on the above Lemma, we establish the following error estimates for the weighted averaging.
 \begin{lemma}\label{lem:inv}
Let $ q_h$ be a finite element function in $V_h$. Then the following inequalities hold:
 \begin{align}
&\|\dgal{(W\bk  q)\cdot \bn}\|^2_{0, \Gamma_K}  \le C_3h\|\bk\|^2\lambda_K\|W\|_{\infty}\|\nabla  q_{h}\|^2_{0, K_1\cup K_2}, \label{eq:l2inv}\\
&   \|\dgal{(W\nabla  q)\cdot \bn}\|^2_{0, \Gamma_K}\le  C_4 h^{-1}\lambda_K\|W\|_{\infty} \|\nabla  q_{h}\|^2_{0, K_1\cup K_2} \label{eq:h1inv}.
\end{align}
\vspace{-0.2in}
\end{lemma}
\begin{proof}
Using  \eqref{eq:kapdef} and \eqref{eq:lamdef}, we can deduce from Lemma \ref{lem:imtr} that
\begin{equation}
\begin{split}
  &\|\dgal{W\bk  q)\cdot \bn}\|^2_{0, \Gamma_K} \\= &\|\kappa_1 (W_1\bk  q_{1,h})\cdot \bn + \kappa_2 (W_2\bk  q_{2,h})\cdot \bn)\|^2_{0, \Gamma_K}\\
  \le & 2\kappa_1^2\|\bk\|^2 \|W_1\|_{\infty}^2\|  q_{1,h} \|^2_{0, \Gamma_K}+2\kappa_2^2\|\bk\|^2 \|W_2\|^2_{\infty} \| q_{2,h}\|^2_{0, \Gamma_K}\\
  \le &C_1\frac{2\kappa_1^2\|\bk\|^2\|W_1\|_{\infty}^2h^2|\Gamma_K|} {|K_1|} \|\nabla q_{1,h}\|^2_{0, K_1}+ \\
  &C_2\frac{2\kappa_2^2\|\bk\|^2\|W_2\|_{\infty}^2h^2|\Gamma_K|}
  {|K_2|} \|\nabla q_{2,h}\|^2_{0, K_2}\\
  =  &2hC_1\kappa_1\|\bk\|^2\|W_1\|_{\infty} \lambda_K \|\nabla q_{1,h}\|^2_{0, K_1}+ \\
  &2hC_2\kappa_2\|\bk\|^2\|W_2\|_{\infty}\lambda_K \|\nabla q_{2,h}\|^2_{0, K_2}\\
  \le & C_3 h\|\bk\|^2 \lambda_K\|W\|_{\infty} \left(  \kappa_1\|\nabla q_{1,h}\|^2_{0, K_1} + \kappa_2 \|\nabla q_{2,h}\|^2_{0, K_2}\right)\\
    \le & C_3h\|\bk\|^2\lambda_K\|W\|_{\infty}\|\nabla  q_{h}\|^2_{0, K_1\cup K_2};
\end{split}
\end{equation}
where we have used the fact $\kappa_i\le 1$ in the last inequality. This completes the proof of inequality \eqref{eq:l2inv} and the inequality \eqref{eq:h1inv} can be established by a similar argument.
\end{proof}

Now, we are ready to show that the bilinear form $a_h(\cdot, \cdot)$  is coercive  and continuous  with respect to the above mesh-dependent norm in the following sense:
\begin{theorem}\label{thm:coer}
Suppose that the stability parameter $\hat{\lambda}$ is large enough. Then there exist two constants $C_5$ and $C_6$ such that
  \begin{align}
&C_5 ||| q_h|||_h^2 \le a_h( q_h, q_h), \quad \forall  q_h \in V_{h, per}; \label{eq:coce}\\
& a_h( q_h, \chi_h) \le C_6 ||| q_h|||_h |||\chi_h|||_h, \quad \forall  q_h, \chi_h \in V_{h, per}.  \label{eq:cont}
\end{align}
\vspace{-0.2in}
\end{theorem}
\begin{proof}
It is noted that \eqref{eq:cont} is a direct consequence of  Lemma \ref{lem:inv}. So we only need to justify the inequality \eqref{eq:coce}. Letting $\phi_h =  q_h$ in \eqref{eq:disbil} and applying the Cauchy-Scharwz inequality and the
 Young's inequality with $\epsilon$, we have
 \begin{equation*}
\begin{split}
  &a_h( q_h ,  q_h) \\
  = &\sum_{i=1}^2\int_{\Omega_i}W(\nabla +\I\bk)  q_h\cdot \overline{(\nabla + \I\bk) q_h}d\bx - \\
  & 2\text{Re}\int_{\Gamma} \dgal{    W(\nabla+\I\bk) q_h\cdot n }\llbracket \overline{ q_h} \rrbracket ds
 +\frac{1}{h}  \| \lambda^{1/2} \llbracket q_h \rrbracket\|^2_{0, \Gamma}\\
 \ge&\|  W^{1/2}(\nabla +\I\bk)  q_h)\|_{0, \Omega_1\cup\Omega_2}^2  - 2 \| \dgal{    W(\nabla+\I\bk) q_h\cdot \bn }\|_{0, \Gamma}
 \|\llbracket  q_h \rrbracket \|_{0, \Gamma} + \frac{1}{h}  \| \lambda^{1/2} \llbracket q_h \rrbracket\|^2_{0, \Gamma}\\
  \ge & C_{u} \|(\nabla +\I\bk)  q_h)\|_{0, \Omega_1\cup\Omega_2}^2  - \sum_{K\in\mathcal{T}_{\Gamma,h}}\frac{h}{\epsilon\lambda_K} \| \dgal{ W(\nabla+\I\bk) q_h\cdot \bn }\|_{0, \Gamma_K}^2 +\\
  &
  \sum_{K\in \mathcal{T}_{\Gamma, h}}\frac{(\hat{\lambda}-\epsilon)\lambda_K}{h}  \| \llbracket  q_h \rrbracket\|^2_{0, \Gamma_K}\\
    \ge & C_{u} \|(\nabla +\I\bk)  q_h)\|_{0, \Omega_1\cup\Omega_2}^2
     -\sum_{K\in\mathcal{T}_{\Gamma,h}} \frac{2h}{\epsilon\lambda_K}\| \dgal{ ( W\nabla q_h)\cdot \bn  }\|_{0, \Gamma_K}^2- \\
     &\sum_{K\in \mathcal{T}_{\Gamma, h}}\frac{2h}{\epsilon\lambda_K}\| \dgal{ (W\bk q_h)\cdot \bn  }\|_{0, \Gamma_K}^2
    +  \sum_{K\in \mathcal{T}_{\Gamma, h}}\frac{(\hat{\lambda}-\epsilon)\lambda_K}{h}  \| \llbracket  q_h \rrbracket\|^2_{0, \Gamma_K}.
    \end{split}
\end{equation*}
    Then, using Lemma \ref{lem:inv}, we can deduce that
 
     \begin{equation*}
\begin{split}
 &a_h( q_h ,  q_h) \\
    \ge &  C_{u} \|(\nabla +\I\bk)  q_h)\|_{0, \Omega_1\cup\Omega_2}^2 +
     \sum_{K\in \mathcal{T}_{\Gamma, h}}\frac{(\hat{\lambda}-\epsilon)\lambda_K}{h}  \| \llbracket  q_h \rrbracket\|^2_{0, \Gamma_K}-\\
   & \sum_{K\in \mathcal{T}_{\Gamma, h}} \frac{2(C_3\|\bk\|^2h^2+C_4)}{\epsilon}\|W\|_{\infty} \|\nabla  q_h\|_{0, K_1\cup K_2}^2   \\
    \ge &  C_{u} \|(\nabla +\I\bk)  q_h)\|_{0, \Omega_1\cup\Omega_2}^2 +
     \sum_{K\in \mathcal{T}_{\Gamma, h}}\frac{(\hat{\lambda}-\epsilon)\lambda_K}{h}  \| \llbracket  q_h \rrbracket\|^2_{0, \Gamma}-\\
   &  \frac{2(16C_3+C_4)}{\epsilon}\|W\|_{\infty} \|  q_h\|_{1, \Omega_1\cup \Omega_2}^2   \\
   \ge & \frac{1}{2} C_{u} \|(\nabla +\I\bk)  q_h)\|_{0, \Omega_1\cup\Omega_2}^2 +
     \sum_{K\in \mathcal{T}_{\Gamma, h}}\frac{(\hat{\lambda}-\epsilon)\lambda_K}{h}  \| \llbracket  q_h \rrbracket\|^2_{0, \Gamma}+\\
   & \left(\frac{1}{2} C_{u} - \frac{2C_I(16C_3+C_4)}{\epsilon}\|W\|_{\infty} \right) \|(\nabla +\I\bk)  q_h)\|_{0, \Omega_1\cup\Omega_2}^2 .
\end{split}
\end{equation*}
Here, $C_I$ is the constant such that $\| q_h\|^2_{1, \Omega_1\cup\Omega_2}\le C_I  \|(\nabla +\I\bk)  q_h)\|_{0, \Omega_1\cup\Omega_2}^2$ for fixed nonzero $\bk$ and we have used the fact $\|\bk\|\le 4$ in the first inequality.
We conclude our proof of \eqref{eq:coce} by taking $\epsilon = \frac{4C_I(16C_3+C_4)}{C_u}\|W\|_{\infty} $ and  choosing  the stability parameter $\hat{\lambda} > \epsilon $.
\end{proof}

Theorem \ref{thm:coer} implies that the finite element eigenvalue  value problem \eqref{eq:fem} is well-posed.
According to the spectral theory,  the discrete eigenvalue of  \eqref{eq:fem} can be enumerated as
\begin{equation}
0 < E^1_h(\bk) \le E^2_h(\bk)  \le \cdots E^{n_h}_h(\bk)
\end{equation}
and the corresponding $L^2$-orthonormal eigenfunctions are $\phi^1_{h}, \phi^2_{h}, \ldots, \phi^{n_h}_h$. Here, $n_h$ is the dimension of the unfitted Nitsche's finite
element space $V_{h, per}$, i.e. $n_h = \dim V_{h, per}$.

The key in the interpolation error estimations of the unfitted Nitsche's methods is to extend a function in the subdomain $\Omega_i$ to the whole domain
$\Omega$.  For any $ q\in H^2(\Omega_i)$,  the  extension operator  of $\phi$ from $H^2(\Omega_i)$ to $H^2(\Omega)$ is denoted by $X_i$
which satisfies
\begin{equation}
(X_i  q)|_{\Omega_i} =  q
\end{equation}
 and
\begin{equation}
\|X_i q\|_{s, \Omega} \le C\| q\|_{s, \Omega_i}, \quad \text{ for } s = 0, 1, 2.
\end{equation}
Let $I_{i,h}$ be the standard nodal interpolation operator  from $C(\overline{\Omega})$ to $V_{i,h}$.
Define the interpolation operator for the finite element space $V_h$ as
\begin{equation}\label{eq:feminterp}
I_h^{\ast} q = (I_{1,h}^{\ast} q_1, I_{2,h}^{\ast} q_2),
\end{equation}
where
\begin{equation}\label{eq:feminterpp}
I_{i,h}^{\ast} q = I_{i,h}X_i q_i, \, i = 1,2.
\end{equation}
For the linear interpolation operator,   \cite{HH2002} established the following optimal error estimates:
\begin{equation}\label{eq:interr}
 \| q  - I_h^{\ast} q \|_{0, \Omega} +  h||| q  - I_h^{\ast} q |||_h \le C h^2 \| q\|_{2, \Omega_1\cup\Omega_2}.
\end{equation}

\subsection{Unfitted Nitsche's method for computing edge modes} In this subsection, we generalize  the unfitted Nitsche's method introduced in previous subsection  to
compute edge modes.  Similarly,  to model the wave propagation  in the heterogeneous media,   we will adopt the jump conditions.   Let $\Gamma_{\Sigma}$
be the union of interfaces in all cells in the  fundamental domain of the cylinder.  Based on this setup,  edge states are the eigenpair of the following
interface eigenvalue problem
\begin{align}
&\mathcal{L}^W\Psi(\bx;\kpar) = E(\kpar)\Psi(\bx;\kpar),   \label{eq:dw_evp_it},\\
&\Psi(\bx+\bv_1;\kpar)=e^{\I\kpar}\Psi(\bx;\kpar),\label{eq:pseudo-per_it}\\
&\Psi(\bx;\kpar) \to 0\ \ {\rm as}\ \ |\bx\cdot\bk_2|\to\infty  \label{eq:localized_it} ,\\
& \left \llbracket \Psi\right \rrbracket  = \left \llbracket W\nabla\Psi\cdot n\right \rrbracket = 0, \quad \text{on } \Gamma_{\Sigma}. \label{eq:edgeinter}
\end{align}
on the infinite domain $\Omega_{\Sigma}$.

 For the interface eigenvalue problem \eqref{eq:dw_evp_it}--\eqref{eq:edgeinter},  the numerical challenges not only  stem from the heterogeneity of the media and the quasi-periodicity of the boundary condition but also stem from  the  infinity nature of the  cylindrical  domain.  For the second difficulty,   thanks to the localization  property of the eigenfunction in the $\bv_2$
  direction,   we can truncate the infinite cylinder  into a finite  computational domain   and replace the localization  condition \eqref{eq:localized_it} by a
  homogeneous Dirichlet boundary condition.    In specific, we define the truncated domain $\Omega_{\Sigma, L}$  as
  \begin{equation}\label{equ:approxdomain}
\Omega_{\Sigma, L} \equiv \left\{ \tau_1\bv_1+\tau_2\bv_2: 0 \le \tau_1 \le 1, -L\le \tau_2\le L\right\}.
\end{equation}

To handle the quasi-periodic boundary condition on $\bv_1$ direction,  we apply the Floquet-Bloch transformation
$\Psi(\bx;   k_{\parallel}) = e^{\I \frac{k_{\parallel}}{2\pi}\mathbf{k}_1\cdot\mathbf{x}}\psi(\mathbf{x};   k_{\parallel})$.
Then, we reformulate  the  problem of  finding edge states as  computing the  eigenpairs of the  interface eigenvalue problem
\begin{align}
&\mathcal{L}^W(k_{\parallel})\psi(\mathbf{x}; k_{\parallel}) = E(k_{\parallel})\psi(\mathbf{x};k_{\parallel}),\label{eq:eigentt} \\
&\psi(\mathbf{x}+\mathbf{v}_1; k_{\parallel}) =\psi(\mathbf{x}; k_{\parallel}), \label{eq:pertt} \\
&\psi(\tau_1\bv_1\pm L\bv_2; k_{\parallel})  = 0, \forall  \, 0 \le \tau_1 \le 1,  \label{eq:inftt}\\
& \left \llbracket \psi\right \rrbracket  = \left \llbracket W(\nabla + \I\bk)\psi\cdot n\right \rrbracket = 0, \quad \text{on } \Gamma_{\Sigma}. \label{eq:intt}
\end{align}
where
\begin{equation}\label{equ:def}
\mathcal{L}^W(k_{\parallel}) = - (\nabla + \I\frac{k_{\parallel}}{2\pi}\mathbf{k}_1)\cdot W(\nabla + \I\frac{k_{\parallel}}{2\pi}\mathbf{k}_1).
\end{equation}

\subsubsection{Unfitted Nitsche's method in a cylinder} To present unfitted Nitsche's method on the truncated  domain $\Omega_{\Sigma, L}$, we introduce the  corresponding Sobolev spaces. Let $W^{k, p}(\Omega_{\Sigma, L})$ denote the Sobolev
spaces of functions defined  on $\Omega_{\Sigma, L}$ with norm $\|\cdot\|_{k, p}$ and
seminorm $|\cdot|_{k, p}$.
To incorporate the  boundary conditions, we define
\begin{equation}
W^{k, p}_{per}(\Omega_{\Sigma,L})\equiv \{ \psi:  \psi \in W^{k, p}(\Omega_{\Sigma, L})  \text{  and   } \psi(\mathbf{x}+\bv_1) = \psi(\mathbf{x})\},
\end{equation}
and
\begin{equation}
W^{k, p}_{per, 0}(\Omega_{\Sigma,L})\equiv \{ \psi:  \psi \in  W^{k, p}_{per}  \text{  and   } \psi(\tau_1\bv_1\pm L\bv_2) = 0 \text{  for  }  0 \le \tau_1\le 1\}.
\end{equation}
When $p=2$, it is simply denoted as $H^k_{per}(\Omega_{\Sigma,L})$ or  $H^k_{per,0}(\Omega_{\Sigma,L})$.

  Note the fact that $\epsilon(\bx)$ is $\Lambda$-periodic.     Then, the computational domain $\Omega_{\Sigma, L}$ can
be split into two disjoint subdomains  $\Omega_{ \Sigma, L}^1$  and $\Omega_{\Sigma, L}^2$, where
\begin{equation}
\Omega_{ \Sigma, L}^i  = \Omega_{\Sigma, L} \cap (\Omega_i+\Lambda),
\end{equation}
for $i = 1, 2$. The restriction of the interface $\Gamma_{\Sigma}$ in $\Omega_{\Sigma, L}$ is denoted by $\Gamma_{\Sigma, L}$,
i.e. $\Gamma_{\Sigma, L} = \Omega_{ \Sigma, L}^1 \cap \Omega_{ \Sigma, L}^2$.   In Figure \ref{fig:gmm}, we give a plot of the interface
$\Gamma_{\Sigma, L}$ with $L = 10$.
\begin{figure}[!h]
   \centering
   \includegraphics[width=\textwidth]{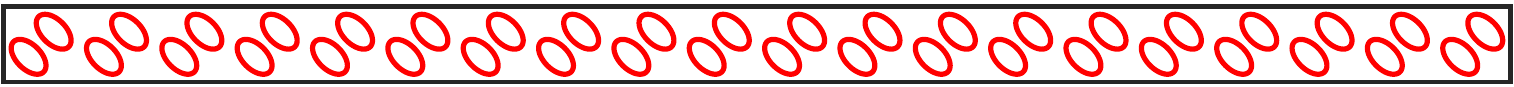}
   \caption{Plot of the interface $\Gamma_{\Sigma, L}$ with $\bv_2$ being the $x$-axis and $\bv_1$ being the $y$-axis.}\label{fig:gmm}
\end{figure}

Let  $\hat{\mathcal{T}}_{ h}$  denote  the uniform triangular partition   of the computational domain  $\Omega_{\Sigma, L}$.
The mesh $\hat{\mathcal{T}}_h$ is generated by  firstly dividing  $\Omega_{\Sigma, L}$ into
$2LN^2$ sub-rhombuses with mesh size $h = \frac{\|\mathbf{v}_1\|}{N}$ and splitting  each sub-rhombus  into two  triangles.
Similarly,  the elements in mesh $\hat{\mathcal{T}}_h$ can be classified  as regular elements or interface elements.
Let $\hat{\mathcal{T}}_{i,h}$ be the set all  elements in $\hat{\mathcal{T}}_h$ covering the subdomain $\Omega_{\Sigma, L}^i$ for $i=1,2$  and $\hat{\mathcal{T}}_{\Gamma,h}$ be the set of interface elements.
The union of all
elements in $\hat{\mathcal{T}}_{i,h}$ is denoted by $\Omega_{\Sigma, L, h}^i$, which is defined as
\begin{equation}
\Omega_{\Sigma, L, h}^i = \bigcup_{K\in \hat{\mathcal{T}}_{i,h}} K, \quad i = 1, 2.
\end{equation}
As demonstrated in the previous section, $\Omega_{\Sigma, L, h}^1$  and $\Omega_{\Sigma, L, h}^2$ form an overlapping decomposition of the computational domain $\Omega_{\Sigma, L}$.

To introduce the finite element space for the unfitted Nitsche's method, we begin with  defining the finite element space on each ficitous subdomain $\Omega^i_{\Sigma, L}$.
Let  $\hat{V}_{i,h}$ be the standard continuous finite element space on $\Omega_{\Sigma, h}^i$ which is defined as
\begin{equation}
\hat{V}_{i,h} = \left\{ v\in C^0(\Omega_{\Sigma, h}^i): v|_{K}\in \mathbb{P}_1(K) \text{ for any } K \in \hat{\mathcal{T}}_{i, h} \right\}, \, i = 1, 2.
\end{equation}
Then, the unfitted Nitsche's finite element space $\hat{V}_h$ is the direct sum of $\hat{V}_{1,h}$ and $\hat{V}_{2,h}$, i.e. $\hat{V}_h  = \hat{V}_{1,h} \oplus \hat{V}_{2,h}$.
To impose the periodic boundary condition in $\bv_1$ direction and homogeneous Dirichlet boundary condition in $\bv_2$ direction,  we introduce the subspace $\hat{V}_{h, 0} = \hat{V}_h \cap  H^k_{per,0}(\Omega_{\Sigma,L})$.

Similar to the previous section, we define  unfitted Nitsche's bilinear form $\hat{a}_h(\cdot, \cdot)$ as
\begin{equation*}
\begin{split}
  \hat{a}_h(u_h ,  v_h) = &\sum_{i=1}^2\int_{\Omega^i_{\Sigma, L}}W(\nabla +\I\frac{k_{\parallel}}{2\pi}\mathbf{k}_1)
  u_h\cdot \overline{(\nabla + \I\frac{k_{\parallel}}{2\pi}\mathbf{k}_1) v_h}d\bx   -\\
  & \int_{\Gamma_{\Sigma, L}}
   \dgal{    W(\nabla+\I\frac{k_{\parallel}}{2\pi}\mathbf{k}_1)u_h\cdot n }\llbracket \overline{ v_h} \rrbracket ds-\\
  &\int_{\Gamma_{\Sigma, L}} \dgal{  \overline{ W(\nabla+\I\frac{k_{\parallel}}{2\pi}\mathbf{k}_1) v_h\cdot n }}\llbracket u_h \rrbracket ds   + \\
 &\frac{1}{h} \int_{\Gamma_{\Sigma, L}} \lambda  \llbracket u_h \rrbracket \llbracket \overline{ v_h} \rrbracket ds,
\end{split}
\end{equation*}
 for any functions $u_h, v_h$ in $\hat{V}_h$.
Then, the unfitted Nitsche's method for the interface  eigenvalue problem  is to find the eigenpair $(\psi_h, E(k_{\parallel}))$ such that
\begin{equation}\label{eq:femnitsche}
 \hat{a}_h(\psi_h,  \eta_h) = E_h(k_{\parallel})\hat{b}(\psi_h,  \eta_h),\quad  \forall  \eta_h \in \hat{V}_{h, 0};
\end{equation}
where
\begin{equation}\label{eq:bhat}
  \hat{b}(\psi_h ,  \eta_h) = \int_{\Omega_{\Gamma}}\psi_h\cdot \overline{ \eta_h}d\bx.
\end{equation}

\subsubsection{Well-posedness of unfitted Nitsche's method in a cylinder}

Using the same argument as in previous subsection, we can prove the unfitted Nitsche's  weak form \eqref{eq:femnitsche}  is consistent in the following sense:
\begin{lemma}
Let  $(\psi, E(k_{\parallel}))$ be the eigenpair of the interface eigenvalue problem \eqref{eq:iteigenLp}--\eqref{eq:incond}.
Then $(\psi, E(k_{\parallel})) \in  H^1_{per,0}(\Omega_{\Sigma,L})\times\mathbb{R}$ also satisfies
\begin{equation}\label{eq:consist_cy}
 \hat{a}_h(\psi,  \eta) = E(k_{\parallel}) \hat{b}(\psi,  \eta),\quad  \forall  \eta \in H^1_{per,0}(\Omega_{\Sigma,L}).
\end{equation}
\vspace{-0.2in}
\end{lemma}

As a direct consequence of the above Lemma, we have the following  Galerkin  orthogonality:
\begin{corollary}
 Let  $(\psi, E(k_{\parallel}))$ be the eigenpair of the interface eigenvalue problem \eqref{eq:eigentt}--\eqref{eq:intt}
 and $(\psi_h, E_h(k_{\parallel}))$ be the corresponding approximate   eigenpair  by the unfitted Nitsche's method. Then we have
 \begin{equation}\label{eq:gal_cy}
 \hat{a}_h(\psi-\psi_h,  \eta_h)  = 0, \quad \forall  \eta_h \in \hat{V}_{h,0}.
\end{equation}
\vspace{-0.2in}
\end{corollary}

We also introduce the following energy norm
\begin{equation}\label{def:energy_norm}
|||\psi|||_h^2 = \| (\nabla + \I\frac{k_{\parallel}}{2\pi}\mathbf{k}_1) \psi\|_{0, \Omega_{ \Sigma, L}^1 \cap \Omega_{ \Sigma, L}^2}^2
+\sum\limits_{K\in \hat{\mathcal{T}}_{\Gamma,h}}h^{-1}\|\llbracket \psi\rrbracket\|_{0,\Gamma_K}^2.
\end{equation}
In term of the energy norm, we shall show that the unfitted Nitsche's bilinear form is coercive and continuous in the following sense
\begin{theorem}\label{thm:cont}
Suppose the stability parameter $\hat{\lambda}$ is large enough. Then there are two constants $C_7$ and $C_8$ such that
  \begin{align}
&C_7 ||| q_h|||_h^2 \le \hat{a}_h( q_h, q_h), \quad \forall  q_h \in \hat{V}_h; \label{eq:coce_cy}\\
& \hat{a}_h( q_h, \chi_h) \le C_8 ||| q_h|||_h |||\chi_h|||_h, \quad \forall  q_h, \chi_h \in \hat{V}_{h,0}.  \label{eq:cont_cy}
\end{align}
\vspace{-0.2in}
\end{theorem}

Theorem \ref{thm:cont} also means the discrete eigenvalue  value problem \eqref{eq:fem} is a well-posed problem.
According to the spectral theory,  the discrete eigenvalue of  \eqref{eq:fem} can be enumerated as
\begin{equation}
0 < E^1_h(k_{\parallel}) \le E^2_h(k_{\parallel})  \le \cdots E^{\hat{n}_h}_h(k_{\parallel})
\end{equation}
and the corresponding $L^2$-orthonormal eigenfunctions are $\psi^1_{h}, \psi^2_{h}, \ldots, \psi^{\hat{n}_h}_h$. Here, $\hat{n}_h$ is the dimension of the unfitted Nitsche's finite
element space $\hat{V}_{h, 0}$.

Likewise, we  use $\hat{X}_i$ to denote  the extension operator for functions defined $\Omega_{\Sigma,L}^i$ to $\Omega_{\Sigma,L}$
which satisfies
\begin{equation}
(\hat{X}_i  \eta)|_{\Omega_i} =  \eta
\end{equation}
 and
\begin{equation}
\|\hat{X}_i \eta\|_{s, \Omega} \le C\| q\|_{s, \Omega_i}, \quad \text{ for } s = 0, 1, 2.
\end{equation}
Let $\hat{I}_{i,h}$ be the standard nodal interpolation operator  from $C(\overline{\Omega_{\Sigma, L}})$ to $\hat{V}_{i,h}$.
Define the interpolation operator for the finite element space $\hat{V}_h$ as
\begin{equation}\label{eq:int_em}
\hat{I}_h^{\ast} q = (\hat{I}_{1,h}^{\ast} q_1, \hat{I}_{2,h}^{\ast} q_2),
\end{equation}
where
\begin{equation}\label{eq:femint_em}
\hat{I}_{i,h}^{\ast} q = \hat{I}_{i,h}\hat{X}_i q_i, \, i = 1,2.
\end{equation}
We can also  show  the following  interpolation  error estimates:
\begin{equation}\label{eq:interr_em}
 \| \eta  - \hat{I}_h^{\ast} \eta \|_{0, \Omega_{\Sigma,L}} +  h||| \eta - \hat{I}_h^{\ast} \eta |||_h \le C h^2 \| \eta\|_{2, \Omega^1_{\Sigma,L}\cup \Omega^2_{\Sigma,L}}.
\end{equation}

\section{Error analysis}\label{sec:symbreak} In this section, we present unified error estimation for the proposed unfitted Nitsche's methods.
Our main analysis tool is the Babuska-Osborn spectral approximation theory \cite{BO1991}.

When we consider the eigenvalue problem \eqref{eq:iteigenLp}--\eqref{eq:incond},  let $A_h(\cdot, \cdot)$  denote the  Nitsche's bilinear function $a_h(\cdot, \cdot)$
which is defined on $V_a: = H^1_{per}(\Omega)$ and $B_h(\cdot, \cdot)$  corresponding  the $L^2$ inner production $b_h(\cdot, \cdot)$
on $V_b: = L^2_{per}(\Omega)$.   Similarly, when we consider the eigenvalue problem \eqref{eq:eigentt}--\eqref{eq:intt},   let $A_h(\cdot, \cdot)$  denote denote the  Nitsche's bilinear function $\hat{a}_h(\cdot, \cdot)$
which is defined on $V_a: = H^1_{per,0}(\Omega_{\Sigma,L})$ and $B_h(\cdot, \cdot)$  corresponding  the $L^2$ inner production $\hat{b}_h(\cdot, \cdot)$
on $V_b: =L^2_{per}(\Omega_{\Sigma,L})$.    The corresponding $L^2$ norm is denoted by $\|\cdot\|_{b}$.  The Nitsche's finite element function  is  denote by $S_h$ which is either $V_{h, per}$ or $\hat{V}_{h,0}$.

For any $f \in V_b$,  let $T: V_b \rightarrow V_a$ be the solution operator for the source problem
such that
\begin{equation}\label{eq:source}
A_h(Tf,  g)  = (f,  g), \quad \forall  g \in V_a.
\end{equation}
We rewrite the  interface eigenvalue problem \eqref{eq:iteigenLp}--\eqref{eq:incond} (or   \eqref{eq:eigentt}--\eqref{eq:intt})  as
\begin{equation}
 T\phi = \mu \phi
\end{equation}
where $\mu = E(\bk)^{-1}$ (or   $\mu = E(k_{\parallel})^{-1}$).   For the source problem \eqref{eq:source}, we can show the following regularity \cite{Ba1970, Ke1975}
\begin{equation}\label{eq:reg}
\|Tf\|_{2, \star} \le C\|f\|_b,
\end{equation}
where the notation  $\|\cdot\|_{2, \star}$ denotes the piecewise $H^2$ norm  $\|\cdot\|_{2, \Omega_1\cup\Omega_2}$ or $\|\cdot\|_{2, \Omega_{\Sigma,L}^1\cup\Omega_{\Sigma,L}^2}$.

Similarly, we introduce the solution operator  $T_h$ for the discrete  source problems which is defined as
\begin{equation}\label{eq:dissource}
a_h(T_{h}f,  g_h)  = (f,  g_h), \quad \forall  g_h \in S_h.
\end{equation}
The unfitted Nitsche's method  \eqref{eq:fem} has the following equivalent representation
\begin{equation}
 T_h\phi_h = \mu_h \phi_h,
\end{equation}
where  $\mu_h = E_h(\bk)^{-1}$   (or   $\mu_h = E_h(k_{\parallel})^{-1}$).
Evidently, both $T$ and $T_h$ are self-adjoint, elliptic, and compact linear operators.

From the interpolation error estimate, we can show  the following error estimates for unfitted Nitsche's method  approximating the
source problem:
\begin{theorem}\label{thm:app}
Let $T$ and $T_h$ be the solution operators defined in \eqref{eq:source} and \eqref{eq:dissource}, respectively. Then we have the following error estimates, for any $f\in L^2(\Omega)$
(or $L^2(\Omega_{\Sigma, L})$),
\begin{align}
& |||Tf-T_hf|||_h \le Ch\|f\|_b, \label{eq:femh1}\\
 & \|Tf - T_hf\|_b  \le Ch^2\|f\|_b. \label{eq:feml2}
\end{align}
\vspace{-0.17in}
\end{theorem}
\begin{proof}
The inequality \eqref{eq:femh1} follows directly  from   Theorem \ref{thm:coer} (or Theorem \ref{thm:cont}), the interpolation error estimate \eqref{eq:interr} (or \eqref{eq:interr_em}),  and the regularity \eqref{eq:reg}.   The inequality \eqref{eq:feml2}
can be proved via the Aubin-Nitsche's tricks, see for example, \cite{HH2002}.
\end{proof}

From the above theorem, we can deduce the following  corollary:

\begin{corollary}
 Let $T$ and $T_h$ be the solution operator defined in \eqref{eq:source} and \eqref{eq:dissource}, respectively.  We have
 \begin{equation}
||T-T_h||_{\mathcal{L}(V_b)}  \le Ch^2.
\end{equation}

 and thus
\begin{equation}
\lim_{h\rightarrow 0} ||T-T_h||_{\mathcal{L}(V_b)}  = 0.
\end{equation}

\end{corollary}
Let $\rho(T)$  (or $\rho(T_h)$) denote the resolvent set of operator $T$ (or $T_h$), and $\sigma(T)$  (or $\sigma(T_h)$ denote the spectrum set of operator $T$ (or $T_h$).
Using the  above approximation property, we have the following   property of  no pollution of the spectrum which is a direct application
of Theorem 9.1 in \cite{Bo2010}: 
\begin{theorem}\label{thm:nopol}
For any compact set $K \subset \rho(T)$, there is $h_0>0$ such that $K\subset \rho(T_h)$ holds for all $h < h_0$.
If $E$ is a nonzero eigenvalue of $T$ with algebraic multiplicity $m$, there are $m$ eigenvalues
$E_h^{1},E _h^{1}, \cdots, E_h^{m}$ of $T_h$ such that all eigenvalues $E_h^{j}, j=1,...,m$ converge to $E$ as $h$ tends to 0.
\end{theorem}

%\emph{Suppose that $E \in \sigma(T)$ is a non-zero eigenvalue with algebraic multiplicity $m$.  Theorem  \ref{thm:nopol} implies that  there are exactly
%$m$ discrete eigenvalue of $T_h$ converging to $E$ as  $h$ tends to zero. \textbf{The italic sentence should be removed since it is a repeat of the theorem. YZ} }

 For any closed smooth curve $\mathcal{C} \subset \rho(T)$  enclosing  $E \in \sigma (T)$
and no other element of $\sigma(T)$,  the Reisz spectral projection associated with $E$ is defined as  \cite{BO1991}
 \begin{equation}
P = \frac{1}{2\pi\I} \int_{\mathcal{C} } (z-T)^{-1}dz.
\end{equation}
When $h$ is sufficiently small,  $\mathcal{C}  \subset \rho(T_h)$ encloses exactly $m$ discrete eigenvalues of $T_h$.   We define analogously the discrete spectral  projection
 \begin{equation}
P_h = \frac{1}{2\pi\I} \int_{\mathcal{C} } (z-T_h)^{-1}dz.
\end{equation}

Thanks to the above preparations, we are ready to show our main eigenpair  approximation results.

\begin{theorem} \label{thm:eigapp}
Let $\mu_h$ be an eigenvalue of $T_h$ such that $\lim_{h\rightarrow 0} \mu_h = \mu$.  Let $g_h$ be a unit eigenvector of $T_h$
 corresponding to the eigenvalue $\mu_h$.  Then there exists a unit eigenvector $g\in R(P)$ such that \textbf{the following estimates hold}
\begin{align}
 &\|g-g_h\|_{0, \Omega} \le  Ch^2 \|g\|_{2, \star},  \label{eq:efunerr}\\
 &  |\mu - \mu_h| \le Ch^2 \|g\|_{2, \star}, \label{eq:evalerr}\\
 & |E - E_h| \le Ch^2 \|g\|_{2, \star}. \label{eq:evalueerr}
\end{align}

\end{theorem}

\begin{proof} \textbf{In order to justify the estimate \eqref{eq:efunerr},  we apply the  Theorem ~7.4 in \cite{BO1991} and the  operator approximation
result \eqref{eq:feml2}, and deduce that}
\begin{align*}
\|g-g_h\|_{0, \Omega} & \le \|(T-T_h)|_{R(P)}\|_b = \sup_{\substack{ q\in R(P)\\
 |\| q|\|_h = 1}} \|T q - T_h q\|_{0, \Omega}
 \le Ch^2 \|\phi\|_{2, \star},
\end{align*}
which completes the proof of \eqref{eq:efunerr}.

Then, we turn to the estimate  \eqref{eq:evalerr}.   Let $v_1$, \ldots, $v_m$ be any basis for $R(P)$. Then, Theorem ~7.3 in \cite{BO1991}  implies that there exists a constant $C$ such that
 \begin{equation}\label{eq:eigapp}
|\mu - \mu_h| \le C \sum_{j,k = 1}^m |((T-T_h)v_j, v_k)| + C\|(T-T_h)|_{R(P)}\|_{0, \Omega}^2.
\end{equation}

To establish upper bound for $|\mu-\mu_h|$, it is sufficient to bound the first term in \eqref{eq:eigapp}. Using  \eqref{eq:source}, \eqref{eq:dissource} and
 the Galerkin orthogonality \eqref{eq:gal} (or \eqref{eq:gal_cy}), we obtain the \eqref{eq:evalerr} by the following calculations
 \begin{equation}
\begin{split}
  ((T-T_h)v_j, v_k) =& (v_j, (T-T_h)v_k)\\
  = &a_h(Tv_j, Tv_k - T_hv_k)\\
  = &a_h(Tv_j-T_hv_j, Tv_k - T_hv_k) + a_h(T_hv_j, Tv_k - T_hv_k)\\
    = &a_h(Tv_j-T_hv_j, Tv_k - T_hv_k) + \overline{a_h(Tv_k - T_hv_k, T_hv_j)}\\
    = &a_h(Tv_j-T_hv_j, Tv_k - T_hv_k)\\
    \le &C|||Tv_j-T_hv_j|||_h |||Tv_k-T_hv_k|||_h\\
    \le& Ch^2\|v_j\|_{2, \star}\|v_k\|_{2, \star}\\
    \le& Ch^2\|g\|_{2, \star}^2.
\end{split}
\end{equation}

The last estimate \eqref{eq:evalueerr} is actually a direct consequence of \eqref{eq:evalerr} by recalling that $\mu=E^{-1}$ (or $\mu_h=E_h^{-1})$. 
%\textbf{Check this is correct or not. Please then delete the following italic sentence.}
%\emph{ Notice the relationship between $\mu$ (or $\mu_h$) and $E$ (or $E_h$) \textbf{We should point this out explicitly. Like ``Recall that $\mu=E^{-1}$ (or $\mu_h=E_h^{-1})$"}.  Then, \eqref{eq:evalueerr} is direct implication of \eqref{eq:evalerr}.}
\end{proof}

\section{Numerical Examples} \label{sec:num}
In this section, we present a series of benchmark numerical examples to verify and validate our theoretical results and demonstrate that the proposed unfitted Nitsche's  methods are effective and efficient numerical methods to compute   the dispersion relation and wave modes for topological materials with very high contrast material weights.

\subsection{Numerical examples for computing dispersion relations}
In this subsection,  we numerically investigate the performance of the unfitted Nitsche's method for computing the dispersion relations of the bulk, i.e. the material weight is $\Lambda$-periodic. We choose the material weight $W$ in  \eqref{eq:materialweight2} with

\begin{equation*}
  \epsilon(\bx)=\left\{\begin{array}{ll}
  1+J, \quad \text{if} \ \bx \in \Omega_1, \\
   1, \,\quad\quad\,\,\,\, \text{if} \ \bb x \in \Omega_2.
    \end{array}\right.
\end{equation*}
The jump ratio of the material coefficient is $(1+J)^2$.   For  large $J$,  we have high contrast material weight.
The radius of $B_r(\bb A)$ and $B_r(\bb B)$ is chosen to be $0.2$.

\subsubsection{Verification of Accuracy}
In this part, we run a series of tests to show the optimal convergence of the numerical eigenvalue obtained by the unfitted Nitsche's method.
To measure the errors, we introduce  the following relative error  of eigenvalues
\begin{equation*}
e_i =  \frac{|E_{i, h_j}(\bk) - E_{i, h_{j+1}}(\bk)|}{ E_{i, h_{j+1}}(\bk)}.
\end{equation*}

\begin{figure} [!h]
   \centering
   \subcaptionbox{\label{fig:ecr20}}
  {\includegraphics[width=0.46\textwidth]{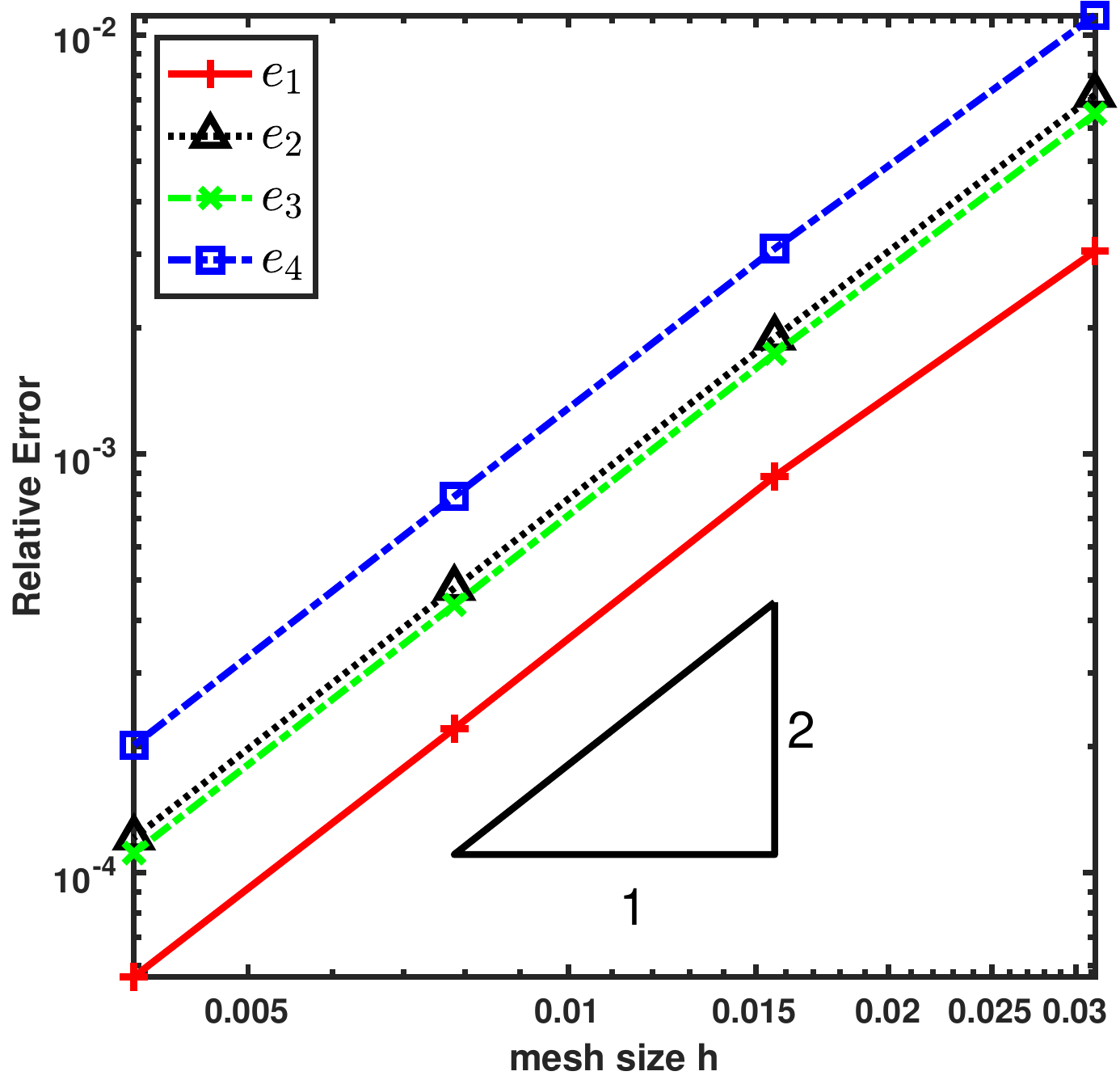}}
  \subcaptionbox{\label{fig:ecr21}}
   {\includegraphics[width=0.46\textwidth]{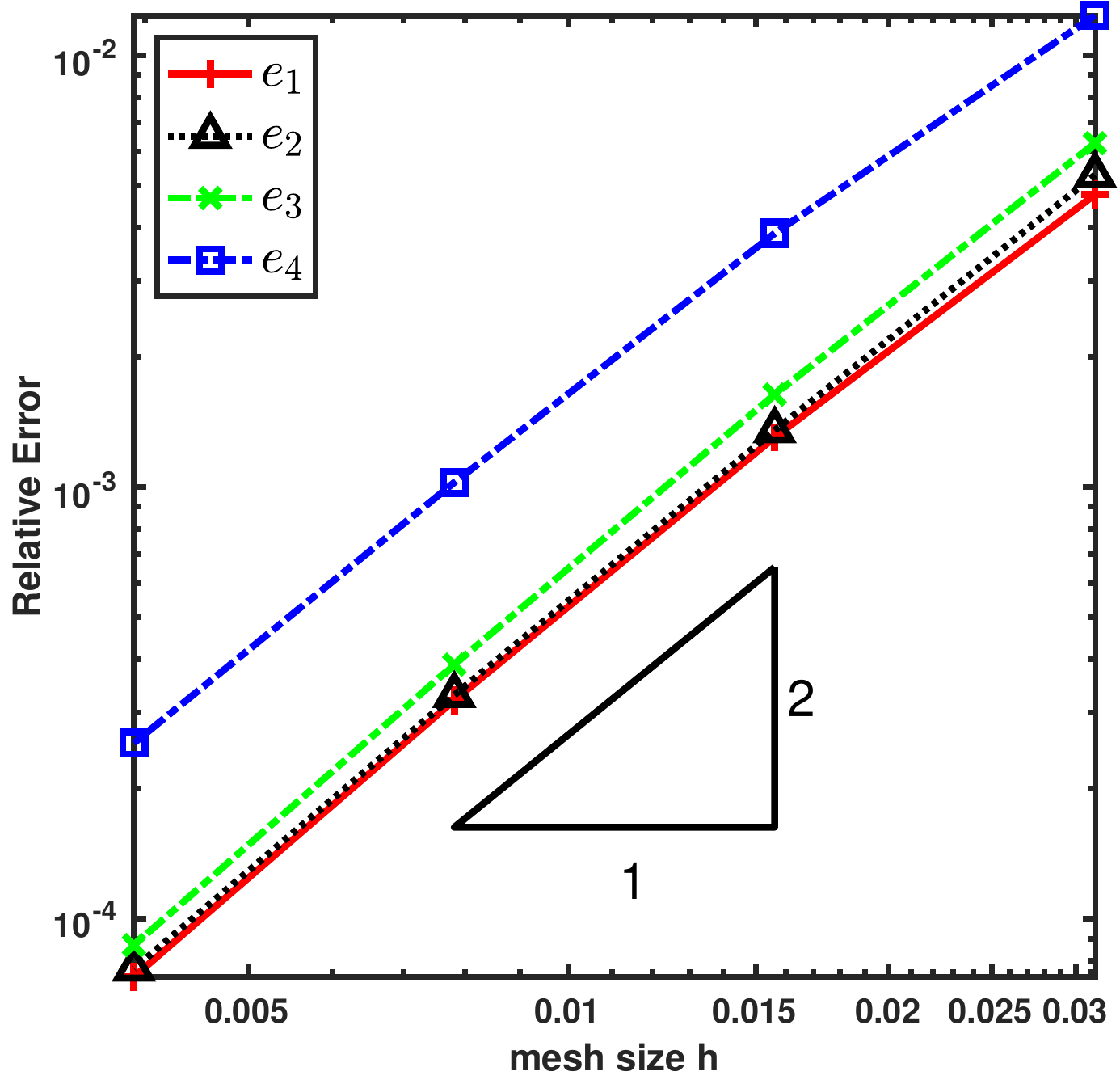}}
   \caption{Numerical errors for eigenvalue approximation: (a) $J=2$ and $\gamma=0$; (b) $J=2$ and $\gamma=0.1$.}
   \label{fig:err2}
\end{figure}

\begin{figure}[!h]
   \centering
   \subcaptionbox{\label{fig:ecr1000}}
  {\includegraphics[width=0.46\textwidth]{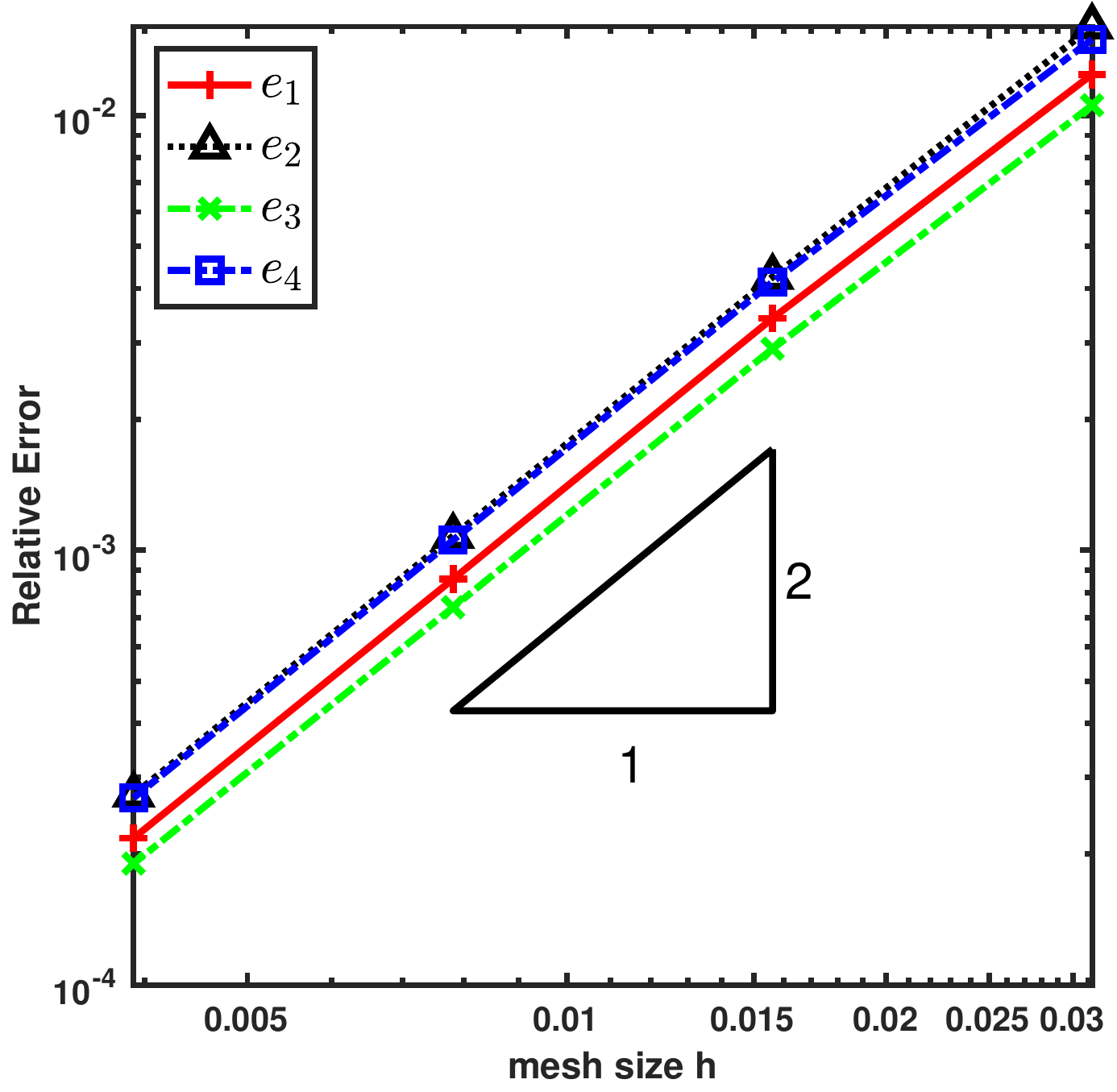}}
  \subcaptionbox{\label{fig:ecr1001}}
   {\includegraphics[width=0.46\textwidth]{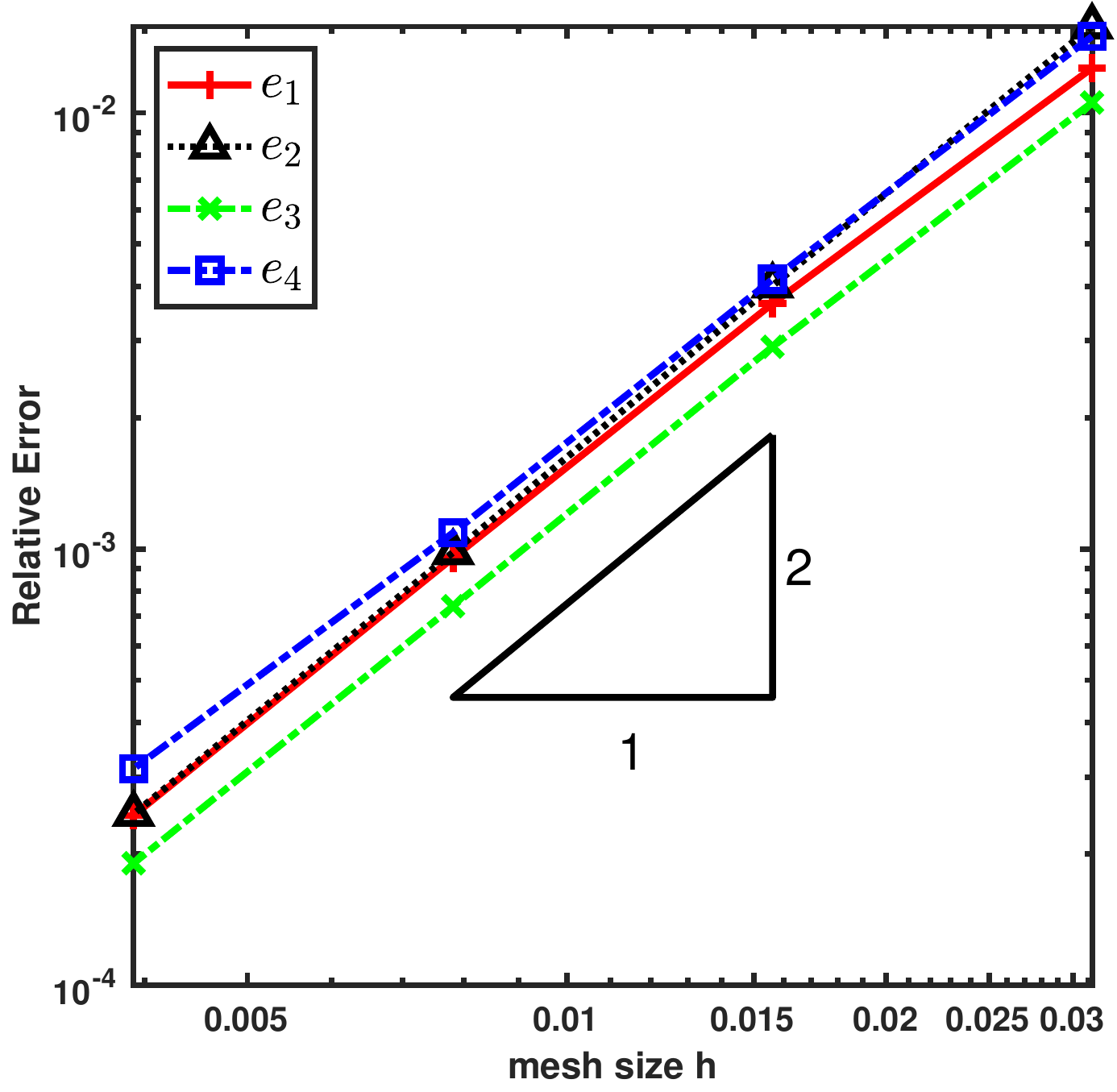}}
   \caption{Numerical errors for eigenvalue approximation: (a) $J=100$ and $\gamma=0$; (b) $J=100$ and $\gamma=0.1$.}
   \label{fig:err100}
\end{figure}

%\begin{figure}[!h]
%   \centering
%   \subcaptionbox{\label{fig:drc2h0}}
%  {\includegraphics[width=0.24\textwidth]{dr_c2_d0}}
%  \subcaptionbox{\label{fig:smc2h0}}
%   {\includegraphics[width=0.24\textwidth]{sm_c2_h0}}
%     \subcaptionbox{\label{fig:drc2h1}}
%  {\includegraphics[width=0.24\textwidth]{dr_c2_d1}}
%  \subcaptionbox{\label{fig:smc2h1}}
%   {\includegraphics[width=0.24\textwidth]{sm_c2_h1}}
%   \caption{Dispersive relation when   $J=2$  (a)  Unfitted Nitsche's Method with $\gamma = 0$; (b) Spectral Method with $\gamma = 0$;
%   (c)  Unfitted Nitsche's Method with $\gamma = 0.1$; (d) Spectral Method with $\gamma = 0.1$.}
%   \label{fig:c1}
%\end{figure}

\begin{figure}[!h]
   \centering
   \subcaptionbox{\label{fig:drc30h0}}
  {\includegraphics[width=0.24\textwidth]{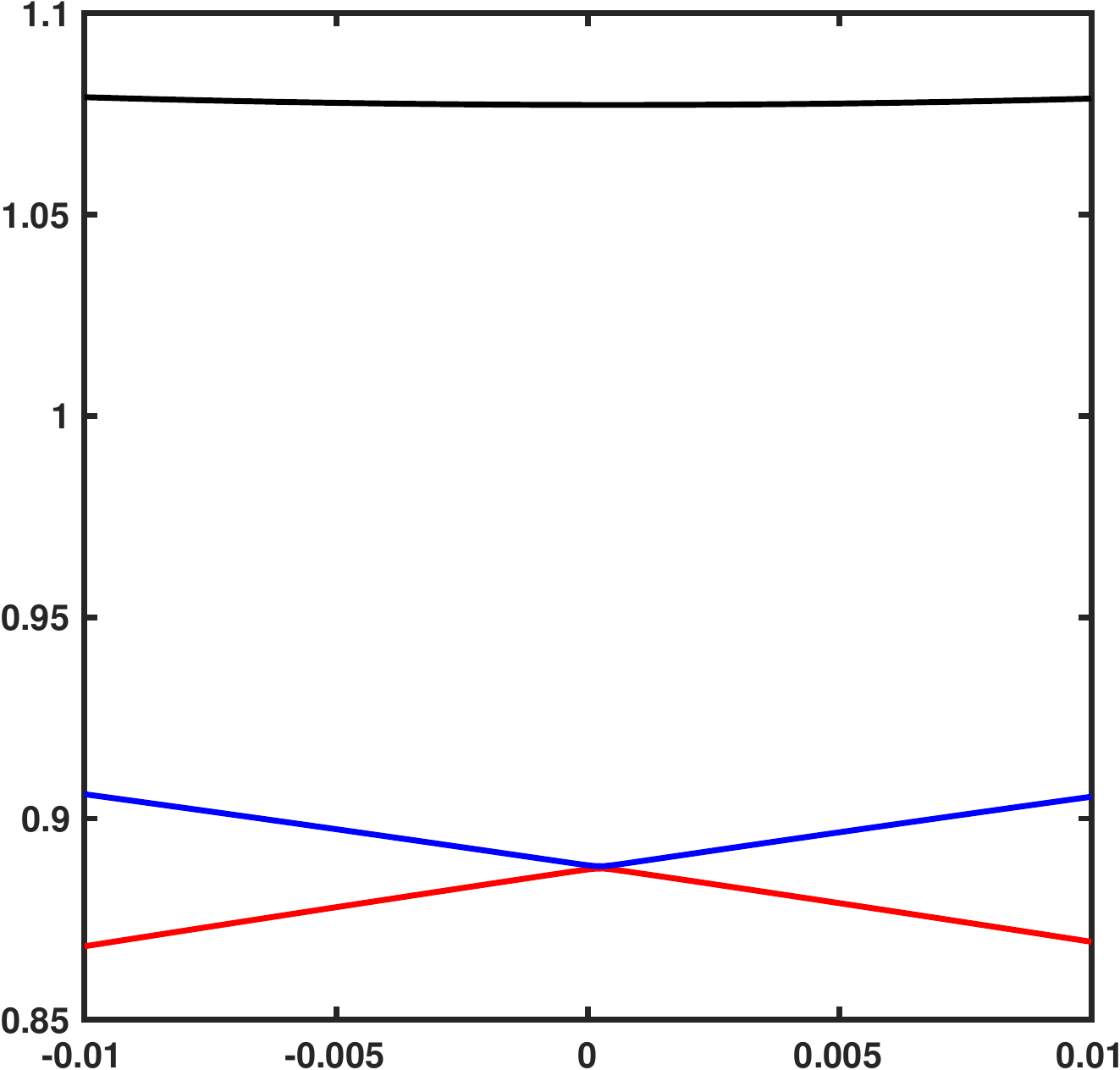}}
  \subcaptionbox{\label{fig:smc30h0}}
   {\includegraphics[width=0.24\textwidth]{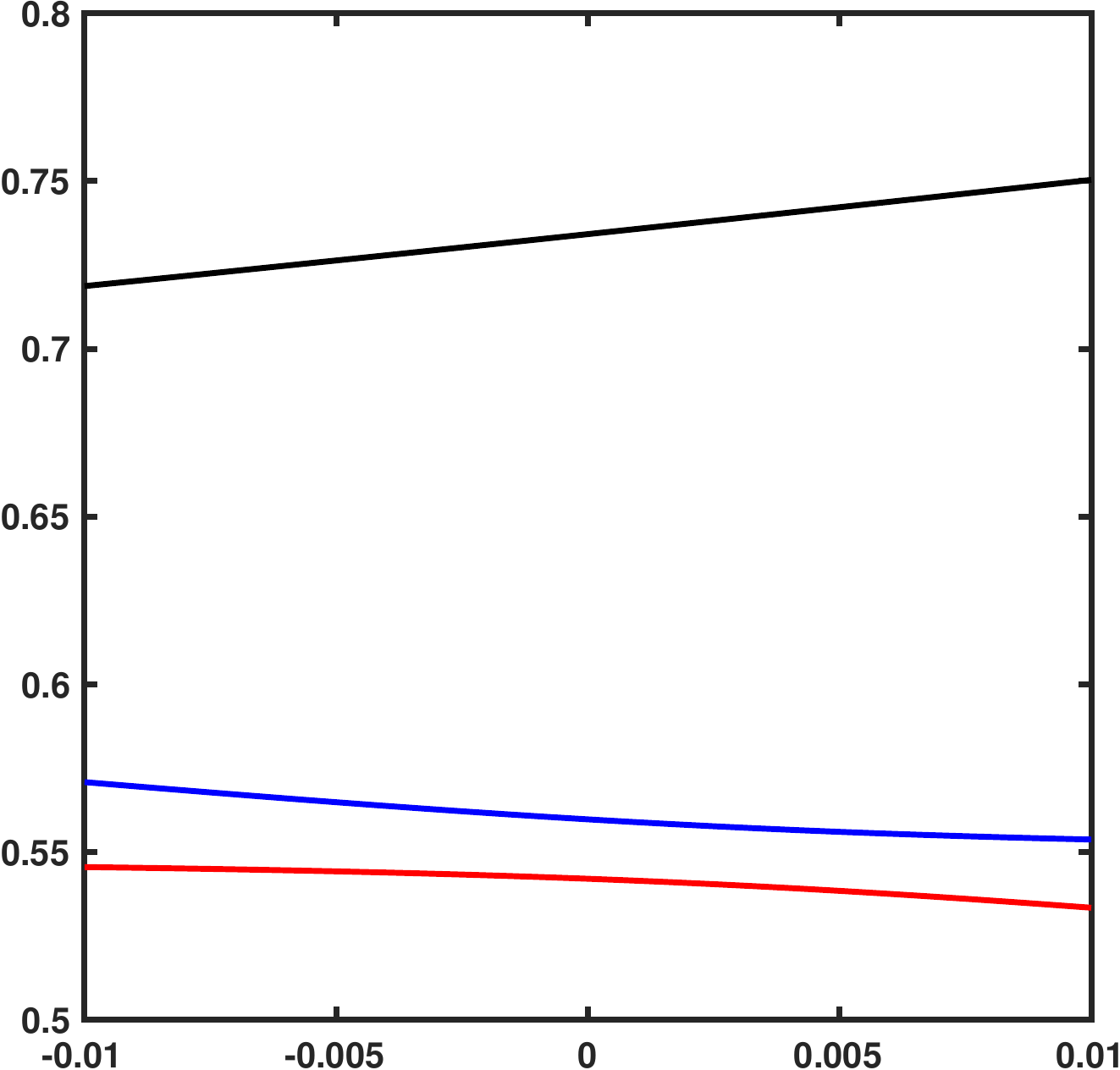}}
      \subcaptionbox{\label{fig:drc30h1}}
  {\includegraphics[width=0.24\textwidth]{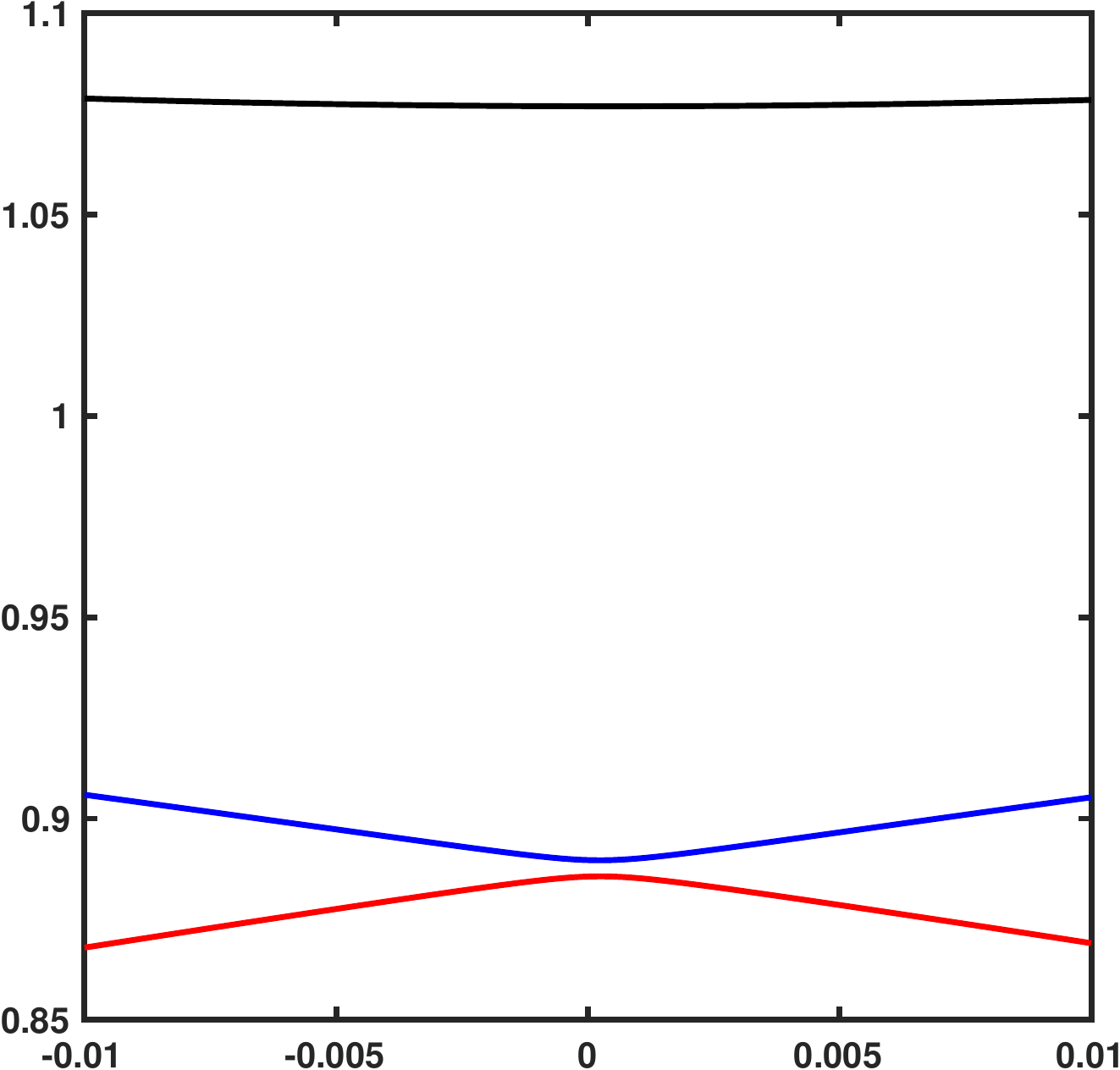}}
  \subcaptionbox{\label{fig:smc30h1}}
   {\includegraphics[width=0.24\textwidth]{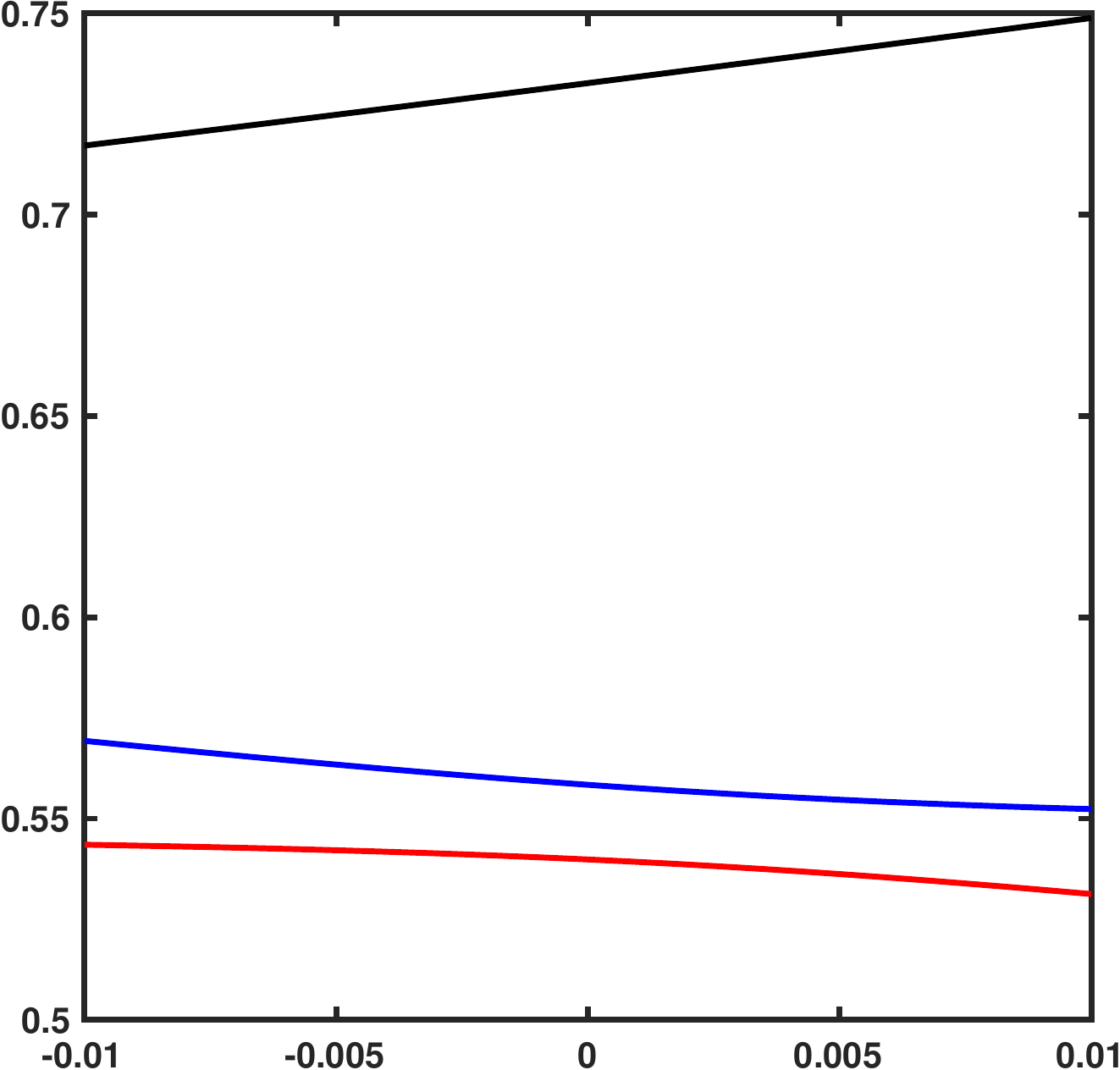}}
   \caption{Dispersion relations when   $J=30$:  (a)  Unfitted Nitsche's Method with $\gamma = 0$; (b) Spectral Method with $\gamma = 0$;
   (c)  Unfitted Nitsche's Method with $\gamma = 0.1$; (d) Spectral Method with $\gamma = 0.1$}
   \label{fig:c30}
\end{figure}

\begin{figure}[!h]
   \centering
   \subcaptionbox{\label{fig:comp_c2h0}}
  {\includegraphics[width=0.47\textwidth]{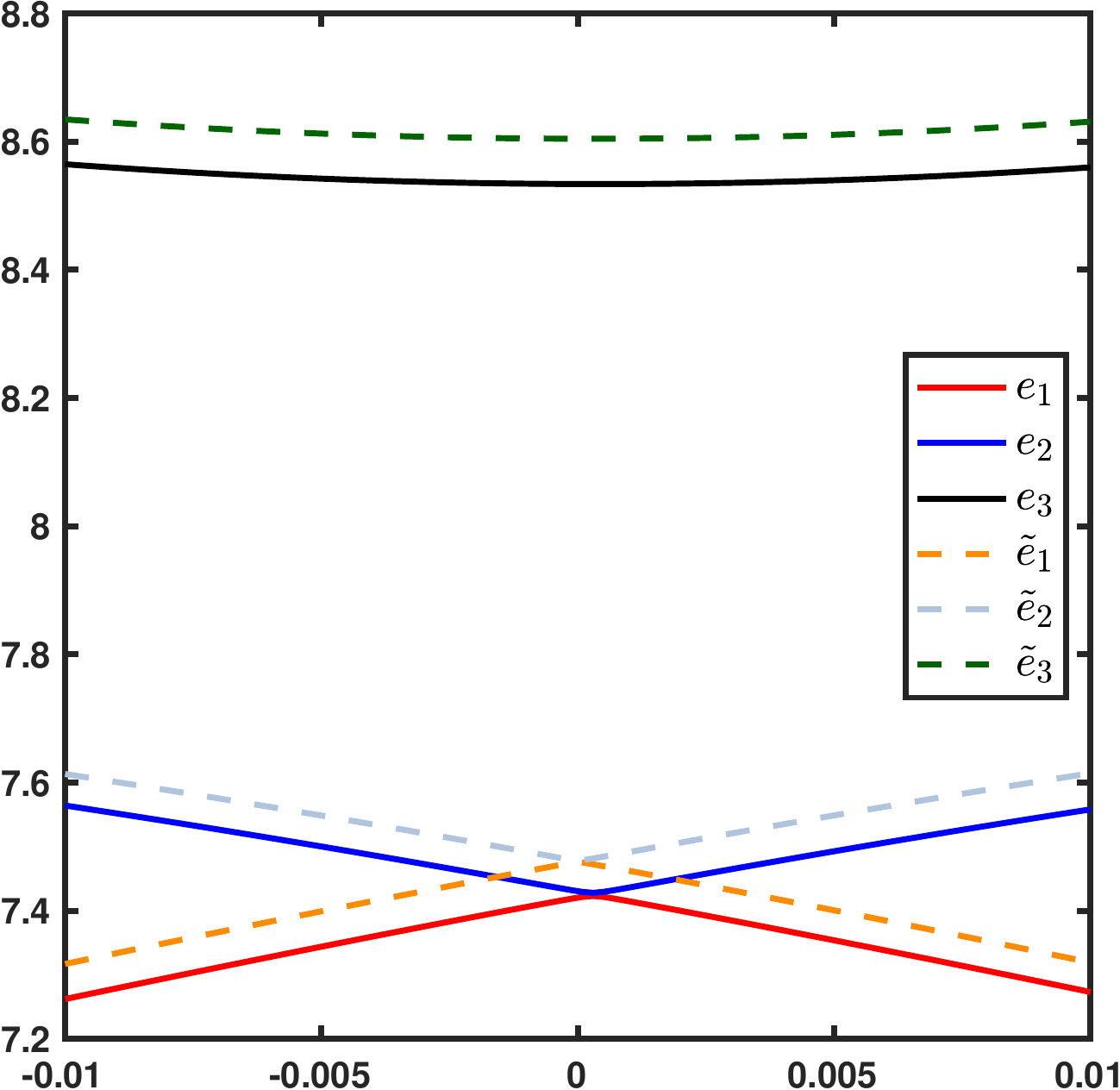}}
  \subcaptionbox{\label{fig:comp_c2h1}}
   {\includegraphics[width=0.47\textwidth]{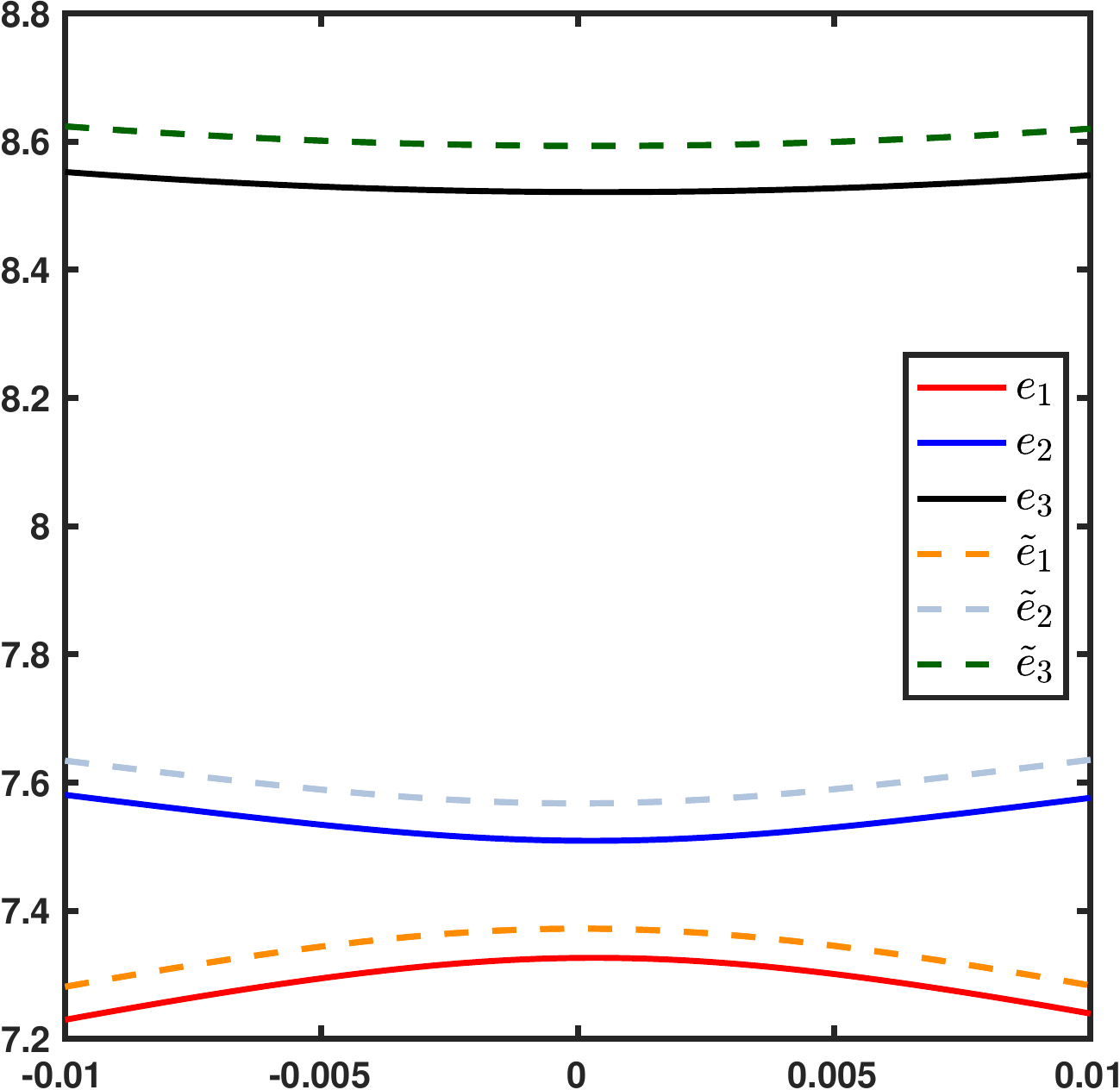}}
   \caption{Comparison of the unfitted Nitsche's method and the Fourier spectral method (solid line: the Unfitted Nitsche's method and dashed line: the Fourier spectral method):  (a) $J=2$ and   $\gamma=0$; (b) $J=2$ and   $\gamma=0.1$}
   \label{fig:comp}
\end{figure}

%\begin{figure}[!h]
%   \centering
%   \includegraphics[width=0.4\textwidth]{band_c2_h1}
%   \caption{Lowest two dispersion surfaces}\label{fig:band}
%\end{figure}
%

Different values of $J$ and $\gamma$ are chosen to test our numerical methods.   The numerical  errors of the first four eigenvalues  are plotted in Figs \ref{fig:err2}--\ref{fig:err100}, where the numerical eigenvalues converge at
the optimal rate $\mathcal{O}(h^2)$.   This confirms that the error estimate for the unfitted Nitsche's method is uniform with respect to the jump ratio $(J+1)^2$.

\subsubsection{Numerical investigation of the dispersion relations}
In this part, we compute the dispersion relations of the bulk and make comparisons with the Fourier spectral methods \cite{skorobogatiy2009fundamentals,LWZ2019}, which expands both the material weight $W(\bx)$ and eigenfunctions in terms of Fourier series.
%For the Fourier spectral method,  we use 512 Fourier modes in each direction to represent the unknown functions.
For the unfitted Nitsche's method, we use the meshes with mesh size $h =\frac{1}{64}$. For the Fourier spectral method,  we use at least 16 Fourier modes in each direction.

%We plot the dispersive relations computed by the unfitted Nitsche's method and
%the Fourier spectral method for  different choices of $J$ and $\gamma$ in Figures \ref{fig:c1h0}--\ref{fig:c30h1}.
%First of all, we consider the case $J=2$.  We plot the dispersive relations computed by the unfitted Nitsche's   and
%the Fourier spectral methods for $\gamma = 0$ and $\gamma=0.1$ in Figure  \ref{fig:c1}.
%In this small jump ratio case, both methods generate correct numerical results.  In particular,  we observe the Dirac point when $\gamma =0$. When $\gamma \neq 0$, we can clearly see that the gap between the first eigencurve and the second eigencurve opens up. The Dirac point disappears.   The observed phenomena are consistent with the theoretical analysis in \cite{LWZ2019}.

We directly consider the case with a relatively large jump ratio with $J=30$.  The numerical results are displayed in Figure  \ref{fig:c30}. For the unfitted Nitsche's method,  we observe the existence of the Dirac point for $\gamma=0$ and the disappearance of the Dirac point when $\gamma=0.1$. This agrees well with the theoretical results \cite{LWZ2019}. Unfortunately, the Fourier spectral method fails to give the correct results.  In specific, we can see that a gap between the first eigencurve and the second eigencurve opens up when $\gamma = 0$ and that the eigencurves are not symmetric which clearly violates the mathematical theory of the spectrum \cite{LWZ2019}. The performance is not improved even when we increase the number of Fourier modes in each direction.

For small jump ratio case, the Fourier spectral method seems to give a reliable result. We will see that our method can do a much better job. To make a  quantitative comparison of those two methods for the small jump ratio case, we graph the results of those two methods in the same plot
when $J=2$ in Figure \ref{fig:comp}. In the Figure,  the numerical results generated by the unfitted Nitsche's method are plotted by solid curves and the numerical results generated by the Fourier spectral method are represented by dashed curves. We can see that the numerical eigenvalues given by the unfitted Nitsche's method are lower than the counterpart given by the Fourier spectral methods. This observation implies that the unfitted Nitsche's method  is much more accurate than the Fourier  spectral method since both methods are
Galerkin methods which give upper bounds of the exact eigenvalues.

In summary, though the Fourier spectral method is widely used in photonic community, it is not a good numerical method to handle the discontinuous material weight, especially when the jump ration is large. In contrast, the unfitted Nitsche's method can give very reliable results in spite of arbitrary large jump ratio.

%The lowest two dispersion surfaces are show in Figure \ref{fig:band} with $J=...$ and $\gamma=0$. We see that two dispersion surfaces intersect at Dirac points.

%
%In Figure \ref{fig:band}, we plot two dispersion surfaces for a honeycomb structured medium,  which intersect at Dirac points for quasi-momenta located at the size vertices of the Brillouin zone.  ({\color{red} That sentence need to be modified since I just take it from \cite{LWZ2019}.})

\subsection{Numerical examples for computing edge modes}
In this subsection, we present numerical examples to show the unfitted Nitsche's method proposed in Section 4 is an efficient numerical method for computing topologically protected edge modes with high contrast material weight and supports the theoretical result for eigenvalue approximation.   We consider the material weight given
in the form
\begin{equation}\label{eq:newform}
W(\bx) =  \epsilon(\bx)^{-1} + \delta \kappa(\delta \mathbf{k}_2\cdot x) \epsilon(\bx)^{-2}\sigma_2,
\end{equation}
where
\begin{equation*}
  \epsilon(\bx)=\left\{\begin{array}{ll}
  1+J,&\text{if} \ \bx \in \Omega_1, \\
   1, &\text{if} \ \bb x \in \Omega_2.
    \end{array}\right.
\end{equation*}
In  \eqref{eq:newform}, $\delta$ is a constant. It is chosen such that the coefficient matrix $W$ is positive definite.
The function $\kappa( \cdot)$ is the transition function (domain wall function) \eqref{stepfun}.

\begin{figure} [!h]
   \centering
   \subcaptionbox{\label{fig:em_c2_d1}}
  {\includegraphics[width=0.47\textwidth]{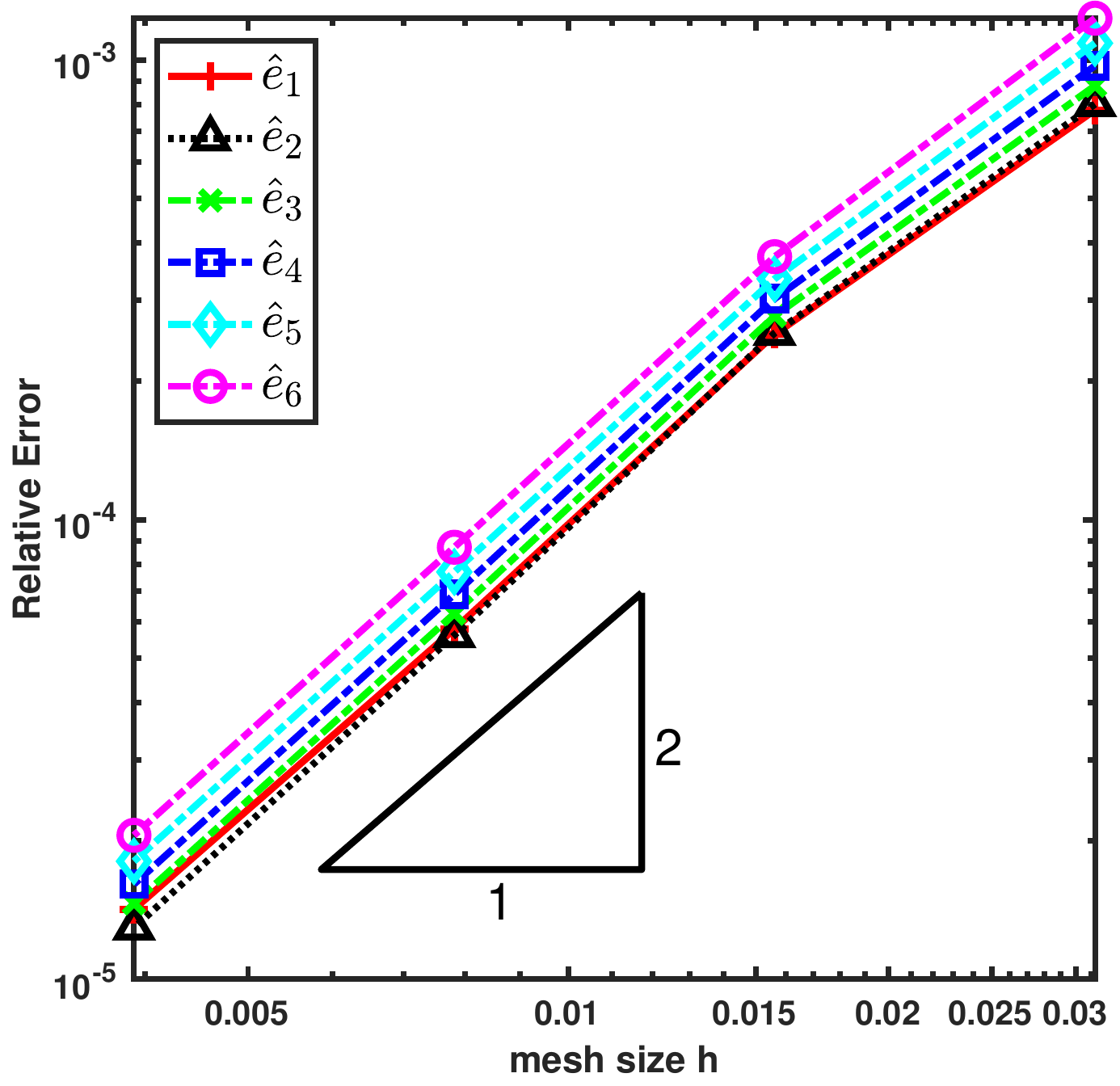}}
  \subcaptionbox{\label{fig:em_c10_d1}}
   {\includegraphics[width=0.47\textwidth]{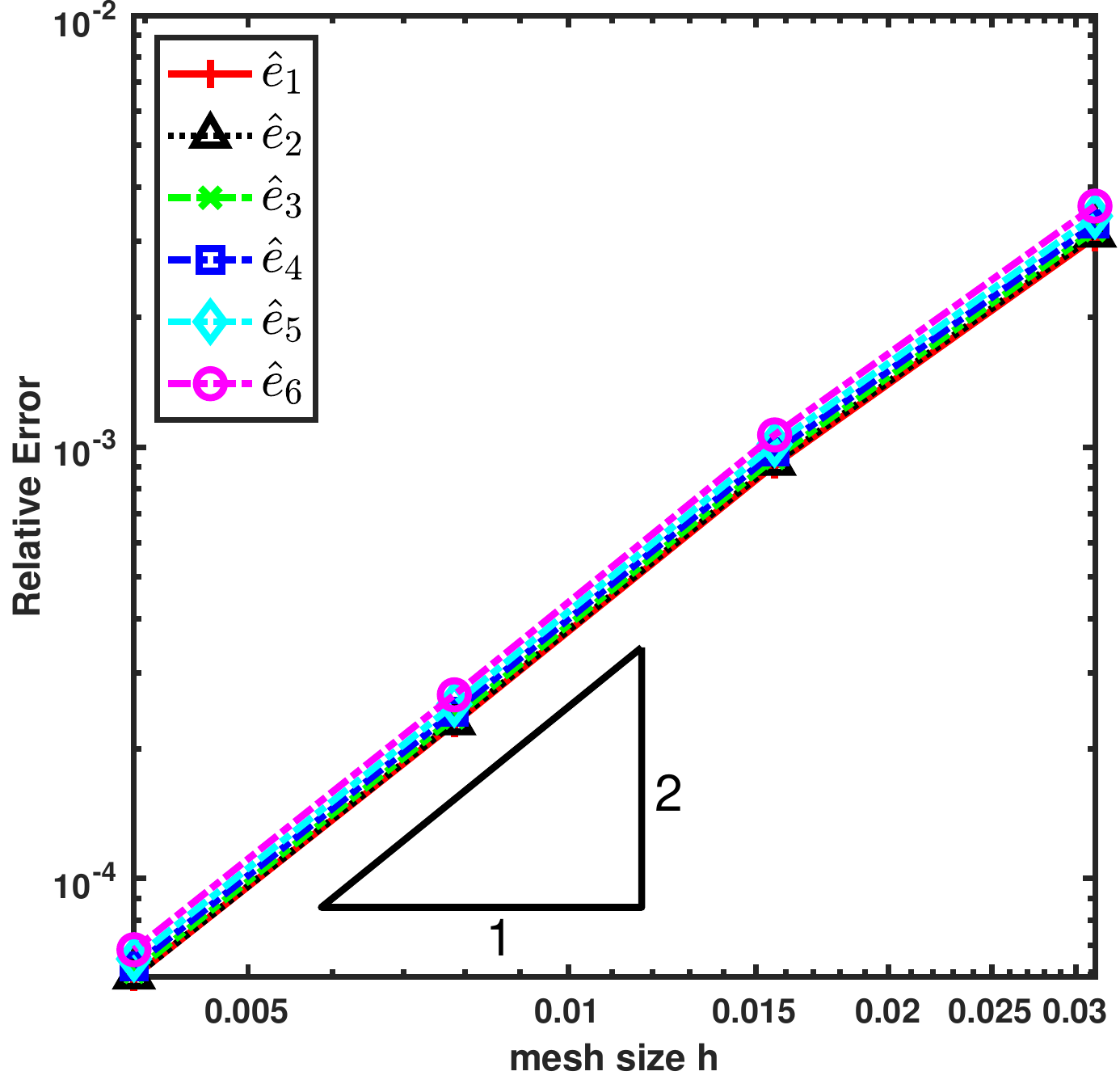}}
   \caption{Numerical errors for eigenvalue approximation: (a) $J=2$ and $\delta=0.1$; (b) $J=10$ and $\delta=0.1$.}
   \label{fig:emerr}
\end{figure}

\subsubsection{Verification of Accuracy} In this  part,  we conduct a benchmark numerical study to verify the optimal convergence of the unfitted Nitsche's method \eqref{eq:femnitsche}.    Similarly, the convergence rate is approximated by the following the relative errors
\begin{equation*}
\hat{e}_i =  \frac{|E_{i, h_j}(k_{\parallel}) - E_{i, h_{j+1}}(k_{\parallel})|}{ E_{i, h_{j+1}}(k_{\parallel})}.
\end{equation*}

In this test, we take $k_{\parallel}  = 0.56\pi$,  $\delta = 0.1$ and $L=10$. We  focus on the computation of the first six eigenvalues.  The numerical results of the convergence test are summarized in Figure \ref{fig:emerr} for $J = 2$ and $J = 10$.  From the data in Figure \ref{fig:emerr}, it is evident that the numerical eigenvalues computed by the unfitted Nitsche's
method \eqref{eq:femnitsche} converges  at the optimal rate $\mathcal{O}(h^2)$.   This is consistent with the theoretical result in the Theorem \ref{thm:eigapp}.

\subsubsection{Computation of the topological edge modes} In this paper, we provide numerical examples to  demonstrate that the proposed unfitted Nitsche's method  is an efficient  method to compute the  topological edge modes in the heterogeneous setting.

\begin{figure}[!h]
 \centering
   \subcaptionbox{\label{fig:em_c2}}
  {\includegraphics[width=0.46\textwidth]{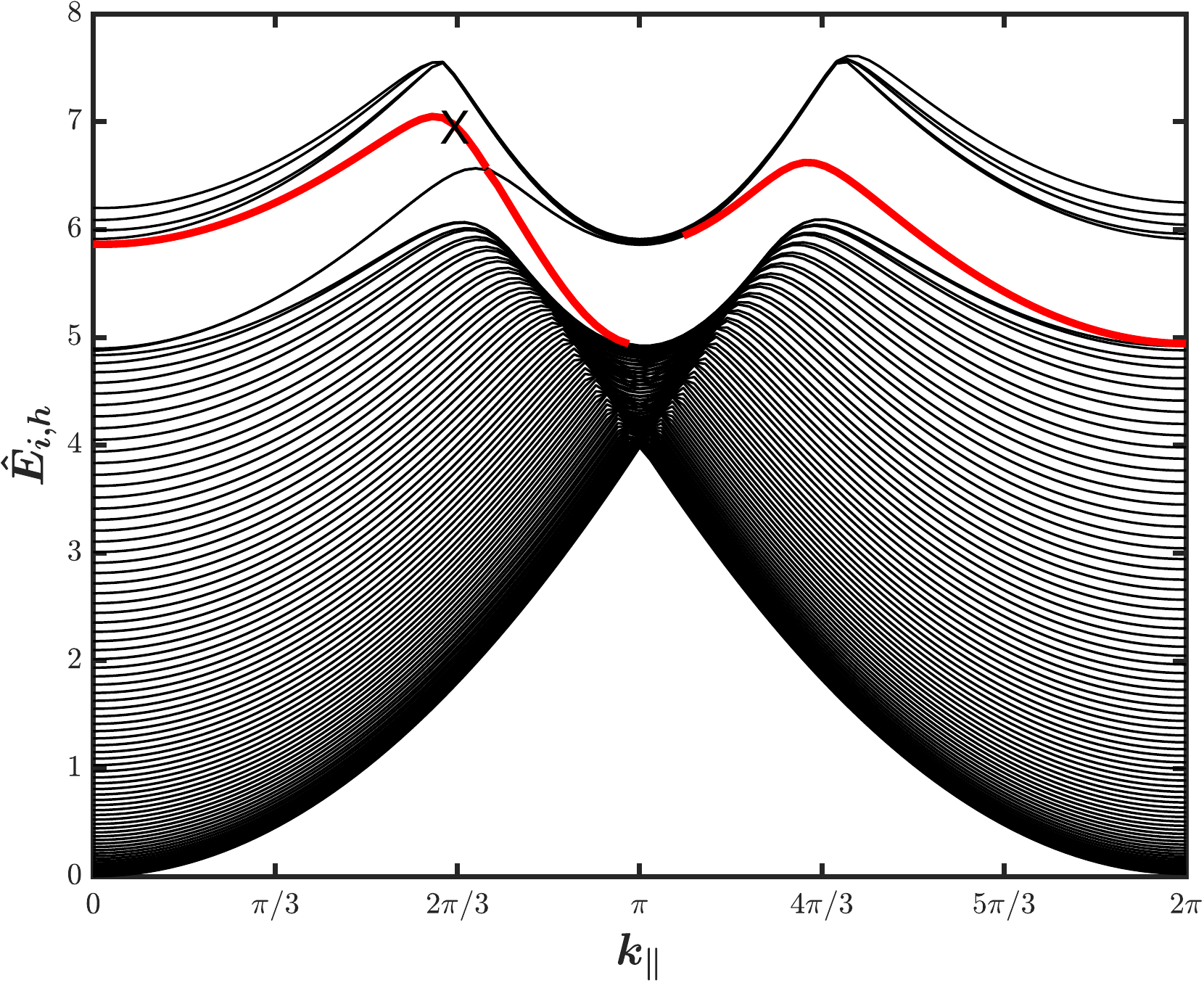}}
  \subcaptionbox{\label{fig:em_c10}}
   {\includegraphics[width=0.47\textwidth]{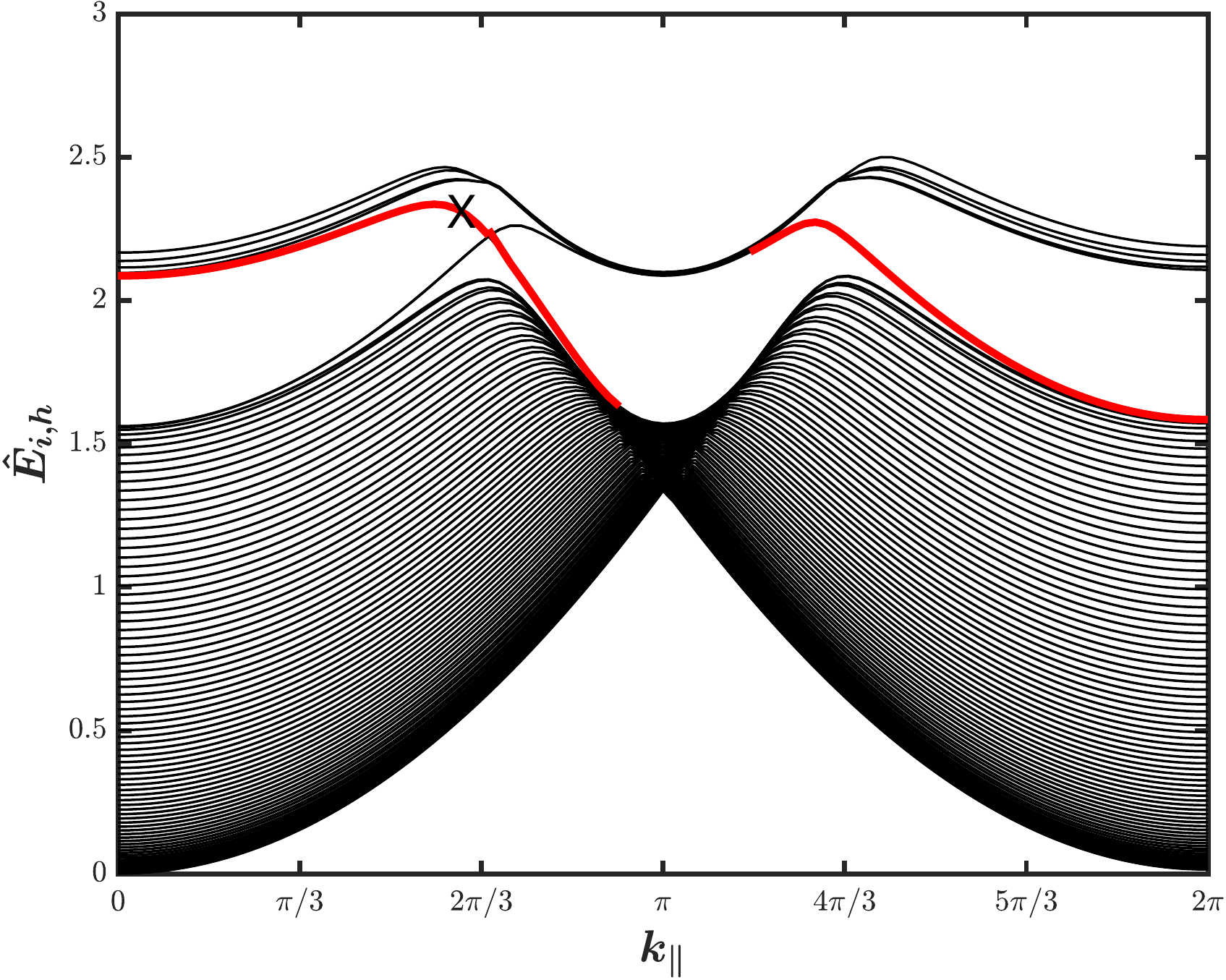}}
   \caption{Plot of the first 85 eigencurves  for with  $L = 80$ where the edge mode is corresponding to the line mark by 'X':  (a) Case $J=2$ and $\delta = 0.6$;
   (b) Case $J=10$ and $\delta = 0.7$. }
   \end{figure}

\begin{figure}[!h]
   \centering
   \subcaptionbox{\label{fig:sol_plot_c2_79}}
  {\includegraphics[width=0.28\textwidth]{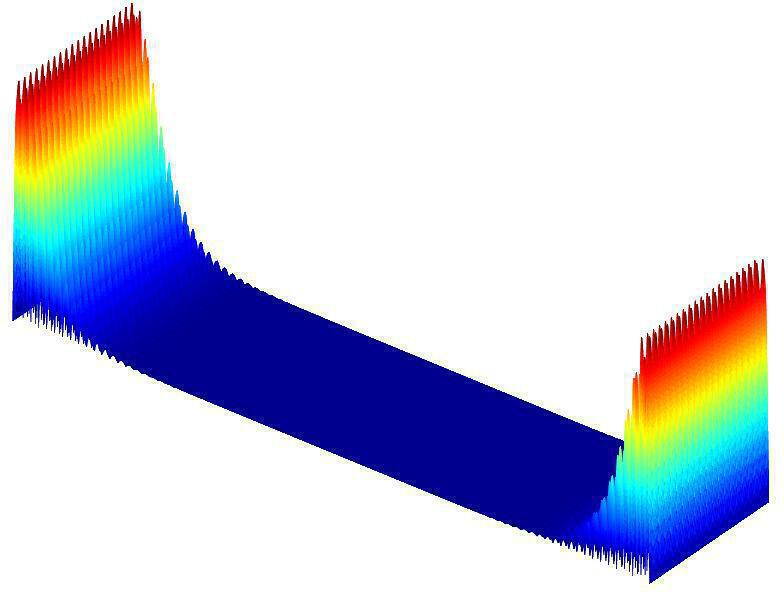}}
  \subcaptionbox{\label{fig:sol_plot_c2_80}}
   {\includegraphics[width=0.28\textwidth]{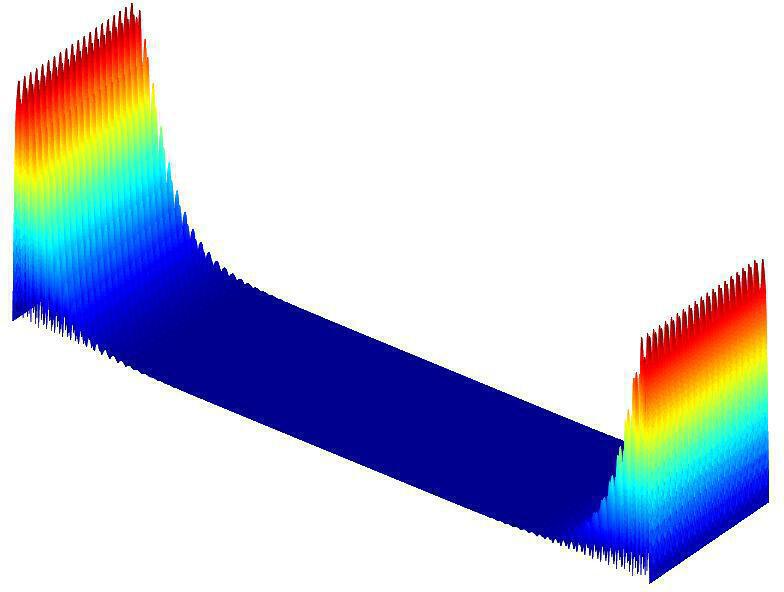}}
   \subcaptionbox{\label{fig:sol_plot_c2_81}}
   {\includegraphics[width=0.34\textwidth]{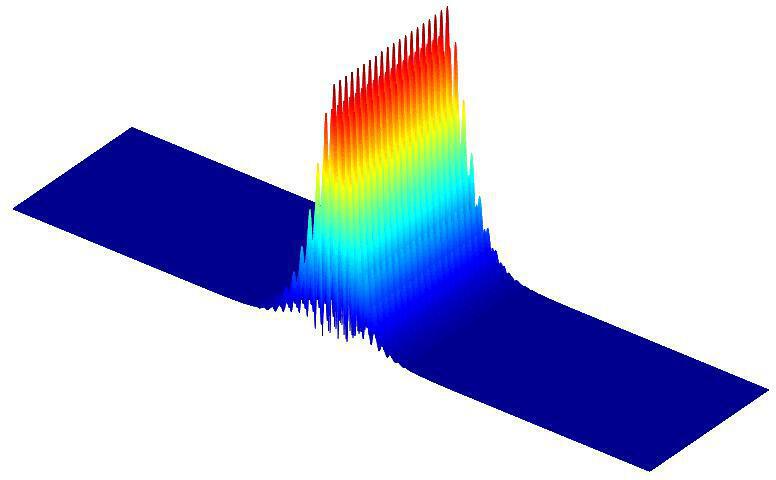}}
   \caption{Plot of the module of the eigenfunctions:  (a) The $79$th eigenfunction; The $80$th eigenfunction; (c) The $81$th eigenfunction}
   \label{fig:c2_sol}
\end{figure}

{\bf Test case  1}.  First of all, we consider the computation of the topological edge states with small jump ratio. In this test,  we choose $J = 2$, $\delta = 0.6$ and
$L = 80$.      In Figure \ref{fig:em_c2}, we show the plot of first 85 eigencurves in term of $k_{\parallel}$.      In Figure \ref{fig:em_c2},  we can see that the red eigencurve
is separated from the eigencurves, which indicates edge states. To demonstrate the existence of edge states, we sketch the modules of the 79th, 80th, and 81th eigenfunctions
at the point $k_{\parallel} = \frac{2\pi}{3}$ in Figure \ref{fig:c2_sol}.  What stands out in the table is that the 81th eigenfunction is located at the center of the computational domain
but the  79th and  80th are both located the boundary of the computational domain.   It suggests that the $81$th eigenfunction is the true edge state of eigenvalue problem \eqref{equ:dw_evp}-\eqref{equ:localized}.  The other two are referred as to pseudo edge state which appear due to fact that the artificial truncation of the computational domain creates two other edges.

\begin{figure}[!h]
   \centering
   \subcaptionbox{\label{fig:sol_plot_c10_79}}
  {\includegraphics[width=0.27\textwidth]{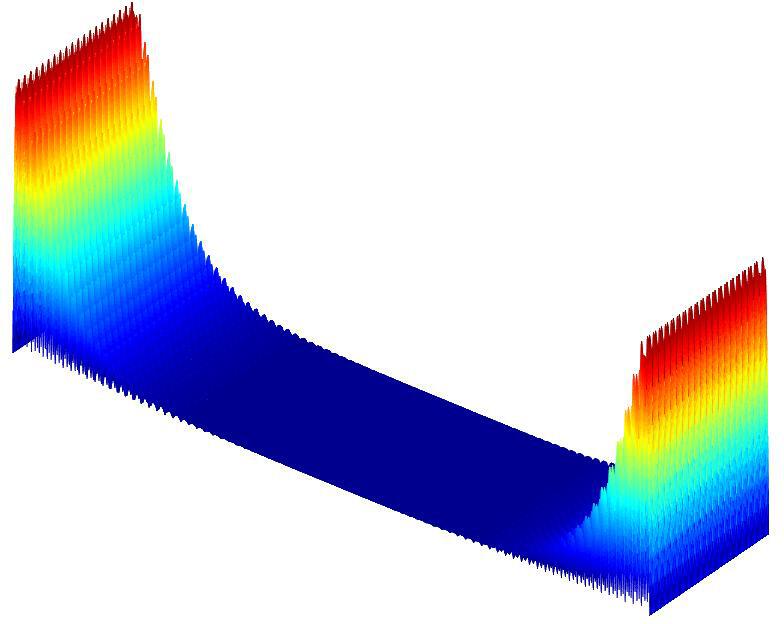}}
  \subcaptionbox{\label{fig:sol_plot_c10_80}}
   {\includegraphics[width=0.27\textwidth]{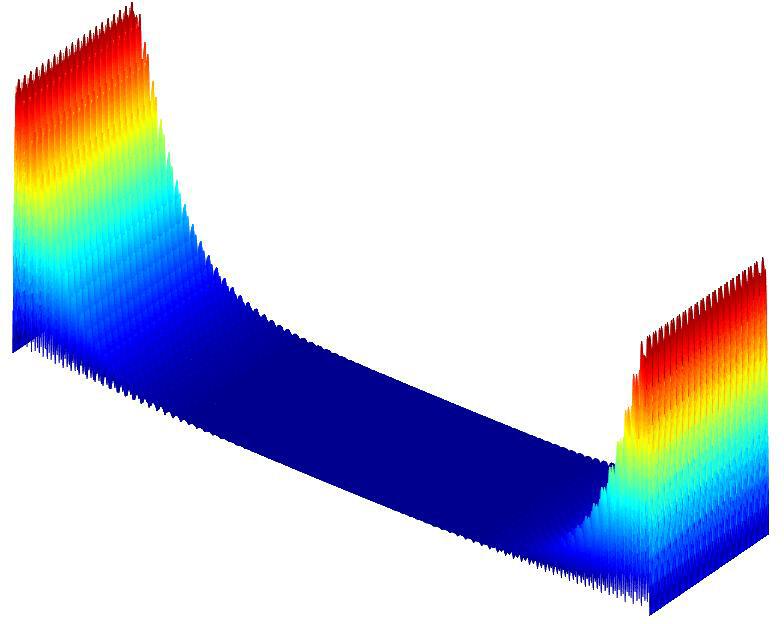}}
   \subcaptionbox{\label{fig:sol_plot_c10_81}}
   {\includegraphics[width=0.35\textwidth]{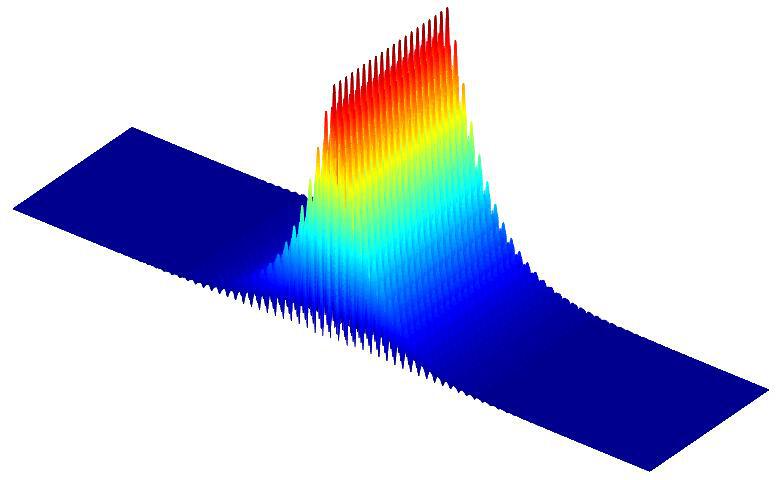}}
      \caption{Plot of the module of the eigenfunctions:  (a) The $79$th eigenfunction; The $80$th eigenfunction; (c) The $81$th eigenfunction}
   \label{fig:c10_sol}
\end{figure}

{\bf Test case  2}.  We consider a relative large jump ratio here.  In this test,  we choose $J = 10$, $\delta = 0.7$ and
$L = 80$.    The plots of eigencurves are presented in Figure \ref{fig:em_c10}. Similarly, we list the plots of the modules of the 79th, 80th, and 81th eigenfunctions
at the point $k_{\parallel} = \frac{2\pi}{3}$ in Figure \ref{fig:c10_sol}.  We observe the same phenomena as in {\bf Test case 1}.  In particular,  we can observe the existence of edge mode.

\section{Conclusion}\label{sec:conclusion}

In this paper, we propose new unfitted Nitsche's methods based on the Floquet-Bloch transform for efficiently simulating photonic graphene with heterogeneous structure. By taking advantage of the structure of underlying meshes, we establish a sharp trace inequality for cut elements, which is the key ingredient to show the stability of the Nitsche's bilinear forms.  The theoretical foundation of the proposed methods builds upon the abstract spectral approximation theory by Babu\v{s}ka and Osborn.   The performance of the proposed unfitted methods is tested with a series of benchmark numerical examples.   Numerical comparison with the Fourier spectral method suggests our method is a better choice for simulating topological materials with discontinuous material weights. In future, we plan to combine the superconvergent tool for unfitted Nitsche's  method in \cite{GY20183} to further improve the accuracy and reduce the CPU time. And we will also apply the results in this work to simulate the evolution of these noval wave modes \cite{hu2020linear,xie2019wave}.

\section*{Acknowledgment}
 The authors thank Professor Michael I. Weinstein for useful discussions. H.G. was partially supported by Andrew Sisson Fund of the University of Melbourne, X.Y. was partially supported by the NSF grant  DMS-1818592, and Y.Z. was partially supported by NSFC grant 11871299.

\appendix
\section{Proof of the Lemma \ref{lem:imtr}}
\label{appendix}
\subsection{A Technical Lemma}
Before giving the proof of Lemma \ref{lem:imtr}, we present a lemma that we shall use.

\begin{lemma}\label{lem:app}
 Let $\bx^j= (x_1^j, x_2^j), ~j = 1, 2, 3$, be the three vertices triangle $K$ and $b_j(\bx)$ be the standard nodal basis function
 associated with $\bx^j$.   Then the following relationship holds
 \begin{equation}\label{eq:basisrel}
|b_j(\bx)| \le 2h|\nabla b_j|, \quad \forall \bx \in K,
\end{equation}
for  $j = 1, 2, 3$.

\end{lemma}
\begin{proof}
 Without loss of generality, we only prove  \eqref{eq:basisrel} for $j = 1$.  Using  the area coordinates \cite{Ci2002},   we have
\begin{equation}
 b_1(\bx) =  \frac{(x_2-x_2^3)(x_1^3-x_1^2)-  (x_1-x_1^2)(x_2^3 - x_2^2) }{2|K|},
\end{equation}
and
\begin{equation}
 \nabla b_i=  \left( \frac{-(x_2^3-x_2^2)}{2|K|},  \frac{(x_1^3-x_1^2)}{2|K|} \right).
\end{equation}
From the above two expressions, we can deduce that
\begin{align*}
  |b_i(\bx)| \le  h\frac{|x_1^3-x_1^2|  +  |(x_2^3 - x_2^2)| }{2|K|} \le 2h\frac{\sqrt{|x_1^3-x_1^2|^2  +  |(x_2^3 - x_2^2)|^2} }{2|K|} = 2h |\nabla b_i|
\end{align*}
where we have used the fact $ |(x_2-x_2^3|\le h$ and
$|x_1-x_1^2|\le h$ for any point $\bx = (x_1, x_2)$ in the triangle $K$.
\end{proof}

\subsection{Proof of Lemma \ref{lem:imtr}}
\begin{proof}
 It is sufficient to show the lemma for the basis functions $b_j$ since $\phi_h$ is a linear combination of $b_j$.  Using Lemma \ref{lem:app}, we can deduce that
 \begin{align*}
  \|b_j\|^2_{0,  \Gamma_T} & \le  |\Gamma_T| \|b_j\|^2_{0, \infty, \Gamma_T} \le |\Gamma_T| \|b_j\|^2_{0, \infty, K_i} \le  4h^2|\Gamma_T|  |\nabla b_i|^2
  =\frac{  4h^2|\Gamma_T|}{|K_i|} \|\nabla b_i\|_{0, K_i}^2;
  \end{align*}
which completes the proof of \eqref{eq:l2tr}.  The inequality \eqref{eq:h1tr} is implied in the above proof.
\end{proof}

\bibliographystyle{siam}
\bibliography{mybibfile,XYpaper}
\end{document}